\documentclass[a4paper,12pt,oneside]{amsart}
\usepackage[top=3cm,left=2.5cm,right=2.5cm,bottom=3cm]{geometry}
\usepackage{hyperref}
\usepackage{amsmath,amsthm,amssymb,comment}
\DeclareSymbolFontAlphabet{\amsmathbb}{AMSb}%
\usepackage{amsfonts,graphicx,color}
\usepackage{enumitem, fancyhdr, dsfont}
\usepackage{thmtools}
\usepackage[nameinlink]{cleveref}
\usepackage[normalem]{ulem}
\usepackage{stmaryrd}
\usepackage{textcomp}
\usepackage{contour}
\usepackage{mathbbol}
\usepackage{stmaryrd}
\usepackage[pagewise]{lineno}%\linenumbers
\usepackage{tikz,float}
\usepackage{tikz-cd}
\usepackage{tasks}

\tikzcdset{row sep/normal=1em}
\tikzcdset{column sep/normal=1.5em}

\hypersetup{
  colorlinks=true,
  linkcolor=caribbeangreen,
  filecolor=caribbeangreen,
  urlcolor=caribbeangreen,
  citecolor=caribbeangreen,
}
\definecolor{caribbeangreen}{rgb}{0.0, 0.8, 0.6}
\setlist{topsep=0ex,itemsep=1ex}

\newcommand{\Awf}{\mathcal{A}}
\newcommand{\Bwf}{\mathcal{B}}
\newcommand{\Cwf}{\mathcal{C}}

\newcommand{\Ewf}{\mathcal{E}}

\newcommand{\Iwf}{\mathcal{I}}
\newcommand{\Jwf}{\mathcal{J}}
\newcommand{\Kwf}{\mathcal{K}}
\newcommand{\Mwf}{\mathcal{M}}
\newcommand{\Nwf}{\mathcal{N}}

\newcommand{\Pwf}{\mathcal{P}}

\newcommand{\Vwf}{\mathcal{V}}
\newcommand{\SN}{\mathcal{SN}}
\newcommand{\Wwf}{\mathcal{W}}
\newcommand{\Rbf}{\mathbf{R}}

\newcommand{\bfrak}{\mathfrak{b}}
\newcommand{\cfrak}{\mathfrak{c}}
\newcommand{\dfrak}{\mathfrak{d}}

\newcommand{\pfrak}{\mathfrak{p}}

\newcommand{\sfrak}{\mathfrak{s}}

\newcommand{\ufrak}{\mathfrak{u}}

\newcommand{\zfrak}{\mathfrak{z}}

\newcommand{\rsub}{\subseteq^\restriction}
\newcommand{\nrsub}{\nsubseteq^\restriction}
\newcommand{\baire}{{}^\omega\omega}

\newcommand{\seq}[2]{\langle#1:#2\rangle}
 \newcommand{\cov}[1]{\mathtt{cov}\hspace{-1pt}\left(#1\right)}
\newcommand{\add}[1]{\mathtt{add}\hspace{-1pt}\left(#1\right)}
\newcommand{\non}[1]{\mathtt{non}\hspace{-1pt}\left(#1\right)}
\newcommand{\cof}[1]{\mathtt{cof}\hspace{-1pt}\left(#1\right)}
\newcommand{\cofJ}[2]{\mathtt{cof}^{#1}\hspace{-1pt}\left(#2\right)}
\newcommand{\addJ}[2]{\mathtt{add}^{#1}\hspace{-1pt}\left(#2\right)}

\newcommand{\Cof}{\mathrm{Cof}}
\newcommand{\id}{\mathrm{id}}
\newcommand{\cf}{\mathrm{cf}}
\newcommand{\RB}{\mathrm{RB}}

\newcommand{\R}{\amsmathbb{R}}

\newcommand{\abs}[1]{\left|#1\right|}%absvalue/cardinality

\newcommand{\leqT}{\preceq_{\mathrm{T}}}

\newcommand{\eqT}{\cong_{\mathrm{T}}}

\newcommand{\set}[2]{\{#1:#2\}}
\newcommand{\largeset}[2]{\left\{#1:#2\right\}}

\contourlength{0.8pt}

\contourlength{0.8pt}

\newcommand{\gen}[1]{\left\langle #1 \right\rangle}

\newcommand{\ED}{{\mathcal{ED}}}

\newcommand{\nwd}{\mathtt{nwd}}
\newcommand{\finfin}{\fin\times\fin}

\newcommand{\fino}{\fin\times\emptyset}
\newcommand{\ofin}{\emptyset\times\fin}
\newcommand{\sel}{\mathrm{Sel}}
\newcommand{\fin}{\text{\rm Fin}}
\newcommand{\pid}[1]{{\mathrm P}(#1)}

\newcommand{\strutt}{\rule[0pt]{0pt}{14pt}}

\title[Relative cofinality of ideals]{Relative cofinality of ideals}

\author{Adam Marton}
\address[Adam Marton]{Department of Applied Mathematics and Business Informatics, Faculty of Economics, Technical University of Košice, B. Němcovej 32, 040 01 Košice, Slovakia}
\email{adam.marton@tuke.sk}

\author{Miroslav Repický}
\address[Miroslav Repický]{Mathematical Institute, Slovak Academy of Sciences, Grešákova 6, 040~01 Košice, Slovakia}
\email{repicky@saske.sk}

\thanks{The first author was supported by the Slovak Research and Development Agency under Contract No.~APVV-20-0045.
The second autor was supported by the grant of the Slovak Grant Agency VEGA 2/0104/24 and
by the Slovak Research and Development Agency under Contract No.~APVV-20-0045.}



\newcommand{\ZFC}{\textbf{\upshape ZFC}}%{\textbf{\itshape ZFC}}
\newcommand{\cons}[1]{\mathrm{Cons}(\ZFC+#1)}

\begin{document}

\newcommand{\unit}{[0,1]}

\makeatletter
\def\@roman#1{\romannumeral #1}
\newcommand{\startlist}{\ \@beginparpenalty=10000}
\makeatother

\numberwithin{equation}{section}
\newcounter{enuAlph}
\renewcommand{\theenuAlph}{\Alph{enuAlph}}

\theoremstyle{plain}
  \newtheorem{theorem}{Theorem}[section]
  \newtheorem{corollary}[theorem]{Corollary}
  \newtheorem{lemma}[theorem]{Lemma}
  \newtheorem{mainlemma}[theorem]{Main Lemma}
  \newtheorem{fact}[theorem]{Fact}
  \newtheorem{prop}[theorem]{Proposition}
  \newtheorem{claim}[theorem]{Claim}
  \newtheorem{hopeth}[theorem]{Hopeful Theorem}
  \newtheorem*{theorem*}{Theorem}
  \newtheorem*{mainthm*}{Main Theorem}
  \newtheorem{teorema}[enuAlph]{Theorem}
  \newtheorem*{corollary*}{Corollary}
  \newtheorem{observation}[theorem]{Observation}
  \newtheorem{conjecture}[theorem]{Conjecture}
\theoremstyle{definition}
  \newtheorem{definition}[theorem]{Definition}
  \newtheorem{notation}[theorem]{Notation}
  \newtheorem{context}[theorem]{Context}
  \newtheorem{assumption}[theorem]{Assumption}
  \newtheorem*{definition*}{Definition}
  \newtheorem*{acknowledgements*}{Acknowledgements}
\theoremstyle{remark}
  \newtheorem{question}[theorem]{Question}
  \newtheorem{problem}[theorem]{Problem}
  \newtheorem{example}[theorem]{Example}
  \newtheorem{remark}[theorem]{Remark}

\def\sectionautorefname{Section}
\def\subsectionautorefname{Subsection}

\begin{abstract}
We introduce a~two-parameter modification of the cofinality invariant of ideals. This allows us to include the interaction of a~pair of ideals in the study of base-like structures. We find the values (cardinal numbers or well-known cardinal invariants) of the invariant for pairs of some critical ideals on $\omega$. We also dichotomously divide pairs of known ideals on the real line based on whether their relative cofinality is trivial or uncountable. Finally, we also study the relative cofinality of maximal ideals.
\end{abstract}

\maketitle

\section{Introduction}\label{sec:int}

Cofinality of an ideal on an infinite set $X$ is the only one among the four classical cardinal characteristics associated with ideals (additivity, covering, uniformity and cofinality) that is interesting even when the underlying set $X$ is countable -- the remaining characteristics become trivial.
In this paper we study the following modification of the cofinality of an ideal, that was originally introduced in \cite{tatra} for ideals on $\omega$.

\begin{definition}
Given ideals $\Iwf$, $\Jwf$ on an infinite set $X$ we define
\[
\cofJ{\Jwf}{\Iwf}=\min\set{\abs{\Awf}}{\Awf\subseteq\Iwf\ \land\ \forall I\in\Iwf\ \exists A\in\Awf\ I\subseteq^\Jwf A}
\]
where $A\subseteq^\Jwf B$ means $A\setminus B\in\Jwf$.
We define $\cofJ{*}{\Iwf}=\cofJ{[X]^{<\omega}}{\Iwf}$.%
\end{definition}

We will call this invariant the \textit{relative cofinality}.
While the relative cofinality can be thought as a~simple generalization of the standard cofinality or $\cofJ{*}{\Iwf}$ (see \cite{inv}), it depends strongly on both parameters.
Therefore, it is rather an invariant capturing the \textit{relationship} between ideals.
Moreover, this invariant is implicitly based on the so-called $\pid{\Jwf}$-property introduced by M.~Ma\v{c}aj and M.~Sleziak \cite{sleziak} and later studied or applied in~\cite{filipow, stanfil, stan, my, repickyqn, repickyqn2}, and is therefore very closely related to ideal convergence.
Indeed, consider the following generalization of the standard $\mathrm{P}$-property:

\begin{definition}[M. Ma\v{c}aj and M. Sleziak]
Let $\Jwf$ be an ideal on an infinite set $X$.
An~ideal $\Iwf$ on $X$ is called a~$\mathrm{P}(\Jwf)$-ideal
if for any sequence $\set{I_n}{n\in\omega}\subseteq\Iwf$ there is $I\in\Iwf$ such that $I_n\subseteq^\Jwf I$ for each $n\in\omega$.
\end{definition}

The definition of $\mathrm{P}(\Jwf)$-ideal is due to \cite{sleziak} but we use the notation from \cite{filipow}.
The relationship ``$\Iwf$ is a~$\pid{\Jwf}$-ideal'' is important for ideal convergence as it describes settings in which some types of ideal convergence are stronger than others.
For instance, Filip\'ow and Staniszewski showed that $\Iwf$-uniform convergence of sequences of real functions is stronger than $(\Iwf,\Jwf)$-quasi-normal convergence if and only if $\Iwf$ satisfies the $\pid{\Jwf}$-property, see \cite{filipow}.
This is an interesting observation, because in this case the relation ``$\Iwf$ is a~$\pid{\Jwf}$-ideal'' describes a~setting in which the relationships between the above-mentioned ideal types of convergence copy the behavior of the classical types of convergence, see \cite{bukovska}.

One can easily prove (see \cite{tatra}) that if $\Iwf$ is a~$\pid{\Jwf}$-ideal, the invariant describing smallest subfamilies of $\Iwf$ ensuring the $\pid{\Jwf}$-property is, in fact, equal to $\cofJ{\Jwf}{\Iwf}$.
Motivated by this fact, in \cite{tatra} the relative cofinality is studied by the first author for pairs of ideals on $\omega$ such that $\Iwf$ satisfies a~$\pid{\Jwf}$-property.
However, the study of the invariant $\cofJ{\Jwf}{\Iwf}$ makes sense not only for the pairs $\Iwf$, $\Jwf$ such that $\Iwf$ is a~$\pid{\Jwf}$-ideal, but for any pair of ideals on the same set $X$.
Thus, we aim to supplement the missing results in Table~2 of \cite{tatra}.
Our second goal is to shift the investigation of the invariant of relative cofinality to uncountable spaces, namely, the unit interval $[0,1]$.
The last goal is to describe the relative cofinality for maximal ideals.

The main results are organized as follows. In \Cref{SecRelativeCofinality} we improve the bounds for non-trivial relative cofinality (i.e., $>1$) found in \cite{tatra}. We find the following natural bounds for it
(\autoref{add<=cof}):
\[
\add{\Iwf}\longrightarrow\addJ{\Jwf}{\Iwf}\longrightarrow\cf(\cofJ{\Jwf}{\Iwf})\longrightarrow\cofJ{\Jwf}{\Iwf}\longrightarrow\cof{\Iwf}.
\]
Regarding the standard cardinal characteristics associated with an ideal~$\Iwf$, we show that the diagram
in \autoref{incRelCof} is complete (\autoref{diagCountEx}), i.e., we show that no arrow can be added to the standard diagram showing inequalities between classical cardinal characteristics associated with an ideal $\Iwf$ and the studied invariant in general.
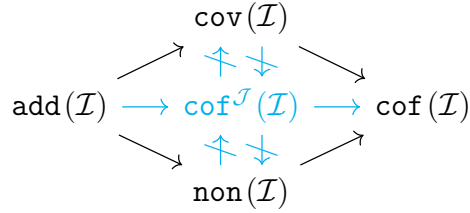
\begin{figure}[h]
$$
\begin{tikzcd}
&\cov{\Iwf}\arrow[dr]\arrow[d, "/" marking, shift left=1.5ex,cyan]\\
\add{\Iwf}\arrow[ur]\arrow[dr]\arrow[r, cyan]&\color{cyan} \cofJ{\Jwf}{\Iwf}\arrow[r, cyan]\arrow[u,"/" marking,shift left=1.5ex,cyan]\arrow[d, "/" marking,shift left=1.5ex,cyan]&\cof{\Iwf}\\
&\non{\Iwf}\arrow[ur]\arrow[u,"/" marking,shift left=1.5ex,cyan]
\end{tikzcd}
$$
\caption{Relationships between the relative cofinality and the classical cardinal characteristics associated with $\Iwf$.
An arrow $\to$ means $\leq$.}
\label{incRelCof}
\end{figure}

In \Cref{SecIdealsOnOmega}, we compute the values missing in \cite[Table~2]{tatra} (\autoref{tablecompletion}), i.e., we complete the list of values of the relative cofinality for pairs of the critical ideals on $\omega$ studied in~\cite{tatra}.
In \autoref{completeTablecofJ}, the contribution of this paper is marked with cyan color, the remaining results are due to \cite{tatra}.
\begin{table}[ht]
$$
\begin{array}{|c|c|c|c|c|c|c|}
\hline\strutt
\big\downarrow\Iwf\ \setminus\overrightarrow{\ \Jwf\ }&\fin &\fino&\ofin&\finfin&\sel&\ED\\[2pt]
\hline\strutt
\fin&1&1&1&1&1&1\\
\fino&{\color{cyan} \omega}&1&{\color{cyan}\omega}&1&{\color{cyan}\omega}&1\\
\ofin&\dfrak&\dfrak&1&1&\dfrak&\dfrak\\
\finfin&{\color{cyan}\dfrak}&\dfrak&{\color{cyan}\omega} &1&{\color{cyan}\dfrak}&\dfrak\\
\sel&{\color{cyan}\cfrak}&{\color{cyan}\cfrak}&1&1&1&1\\
\ED&{\color{cyan}\cfrak}&{\color{cyan}\cfrak}&{\color{cyan}\omega}&1&{\color{cyan}\omega}&1\\[2pt]
\hline
\end{array}
$$
\caption{Cardinality $\cofJ{\Jwf}{\Iwf}$ for all pairs of ideals on $\omega\times\omega$ considered in this work.}
\label{completeTablecofJ}
\end{table}

\Cref{SecIdealsOnUnit} is devoted to the ideals $\Nwf$, $\Mwf$, $\nwd$, $\Ewf$, $\SN$ on $[0,1]$.
We show in which cases $\cofJ{\Jwf}{\Iwf}$ has a~non-trivial value.
We prove the dichotomy shown in \autoref{table2}.
\begin{table}[ht]
$$
\newcommand{\0}{\phantom{>{}}}
\begin{array}{|c|c|c|c|c|c|}
\hline\strutt
\big\downarrow\Iwf\ \setminus\overrightarrow{\ \Jwf\ }&\Nwf &\Mwf&\nwd&\Ewf&\SN\\[2pt]
\hline\strutt
\Nwf&\01&1&>1&>1&>1\\
\Mwf&\01&1&>1&>1&>1\\
\nwd&>1&1&\01&>1&>1\\
\Ewf&\01&1&>1&\01&>1\\
\SN&\01&\text{independent}&>1&\text{independent}&\01\\[2pt]
\hline
\end{array}
$$
\caption{Dichotomous behavior of the relative cofinality on reals.
``Independent'' means that the corresponding case is not decidable within $\ZFC$.}
\label{table2}
\end{table}
We improve the inequality $\cofJ{\Nwf}{\nwd}\geq\add{\nwd}=\aleph_0$
to $\cofJ{\Nwf}{\nwd}\geq\dfrak$ (\autoref{uncountableCof}).

In \Cref{secMaxIdeals} we study relative cofinality of maximal ideals. We show that $\cofJ{\Jwf}{\Iwf}=1$ if and only if $\cofJ{\Jwf}{\Iwf}\leq\omega$ (\autoref{maxB}). We also show (\autoref{maxC}) that, in a~sense, the uniformity of a~maximal ideal $\Iwf$ is the boundary between triviality and uncountability of the relative cofinality $\cofJ{[X]^{<\kappa}}{\Iwf}$. More precisely, we have shown that $1=\cofJ{[X]^{<\lambda}}{\Iwf}<\aleph_1\leq\cofJ{[X]^{<\kappa}} {\Iwf}$ whenever $\kappa\leq\non{\Iwf}<\lambda$. A~new cardinal invariant bounding relative cofinality of maximal ideals from below in specific cases is briefly studied.

\section{Preliminaries}
We work in the standard set theory $\ZFC$.
We let $\omega$ denote the set of all natural numbers and we will use $\forall^\infty n\in\omega$ to abbreviate ``for all but finitely many natural numbers'' and $\exists^\infty n\in\omega$ ``for infinitely many natural numbers''.

\begin{definition}
Let $X$ be an infinite set.
A family $\Iwf\subseteq\Pwf(X)$ is an \textit{ideal on} $X$ if for all $A,B\subseteq X$,
\begin{enumerate}[label=(\arabic*)]\label{idealdef}
\item
if $A\in\Iwf$, $B\subseteq A$, then $B\in\Iwf$,
\item
if $A,B\in\Iwf$, then $A\cup B\in\Iwf$,
\item
$X\notin\Iwf$ and $[X]^{<\omega}\subseteq\Iwf$.
\end{enumerate}
Moreover, if $\Iwf$ satisfies also $\aleph_1$-completeness, i.e., the condition
\begin{enumerate}
\item[($2^*$)]
if $\set{I_n}{n\in\omega}\subseteq\Iwf$, then $\bigcup_{n\in\omega}I_n\in\Iwf$,
\end{enumerate}
then the ideal $\Iwf$ is called a~$\sigma$\textit{-ideal}.
\end{definition}

Of particular importance are two kinds of ideals: ideals on $\omega$ (i.e., on countable sets) and ideals on $\R$ (or other Polish spaces equipped with an adequate measure).
Both of these types are an important tool in the study of the structure of the real line.

A family $\Bwf\subseteq\Iwf$ is a~\textit{base of} $\Iwf$ if $\Bwf$ is cofinal in $\Iwf$ partially ordered by inclusion, i.e., if $\forall I\in\Iwf$ $\exists B\in\Bwf$ $I\subseteq B$.
For $\Awf\subseteq\Pwf(X)$ we define
\[
\gen{\Awf}=\largeset{E\in\Pwf(X)}{\exists\Awf'\in[\Awf]^{<\omega}\ E\subseteq^*\bigcup\Awf'}.
\]
If $X\notin\gen{\Awf}$ then $\gen{\Awf}$ is the smallest ideal containing $\Awf$ and is called the ideal \textit{generated by}~$\Awf$.
We use $\gen{A}$ instead of $\gen{\{A\}}$.
If $\Iwf$, $\Jwf$ are ideals on $X$ we let $\Iwf\vee\Jwf=\gen{\Iwf\cup\Jwf}$.
Note that $\Iwf\vee\Jwf=\set{I\cup J}{I\in\Iwf,J\in\Jwf}$ and if
it is an ideal, then it is the supremum of $\{\Iwf,\Jwf\}$ in the family of all ideals on $X$ ordered by inclusion.

Ideals
$\Iwf$ and $\Jwf$ are said to be \textit{orthogonal}, written $\Iwf\perp\Jwf$, if there is $I\in\Iwf$ such that $X\setminus I\in\Jwf$.
We define the family of $\Iwf$-positive sets by $\Iwf^+=\Pwf(X)\setminus\Iwf$ and the dual filter to $\Iwf$ by $\Iwf^*=\set{X\setminus I}{I\in\Iwf}$.
If $E\in\Iwf^+$ we define the following ideal on $E$: $\Iwf{\restriction}E=\set{I\cap E}{I\in\Iwf}$.
For technical reasons we will always consider the set $\{\emptyset\}$ as an ideal (even $\sigma$-ideal) on any infinite set $X$, even if this set does not satisfy the third condition of \autoref{idealdef}.

Let $\Iwf$ be an ideal on $X$.
We consider the four \textit{cardinal characteristics associated with} $\Iwf$: additivity, covering, uniformity and cofinality, which are defined as follows.
\begin{align*}
&\add{\Iwf}=\min\largeset{\abs{\Awf}}{\Awf\subseteq\Iwf
\text{ and }\bigcup\Awf\notin\Iwf},\\
&\cov{\Iwf}=\min\largeset{\abs{\Awf}}{\Awf\subseteq\Iwf\text{ and }\bigcup \Awf=X},\\
&\non{\Iwf}=\min\set{\abs{A}}{A\subseteq X\text{ and }A\notin\Iwf}, \\
&\cof{\Iwf}=\min\set{\abs{\Bwf}}{\Bwf\text{ is a~base of }\Iwf}.
\end{align*}

One can easily prove the inequalities between the cardinal characteristics in \autoref{basicineq}.
If $\Iwf$~is a~$\sigma$-ideal, the leftmost part of the diagram can be changed to
$\aleph_1\longrightarrow\add{\Iwf}$.
\begin{figure}[h]
$$
\begin{tikzcd}
&&\cov{\Iwf}\arrow[r]\arrow[ddr]&\cof{\Iwf}\arrow[dr]\\
\aleph_0 \arrow[r] & \add{\Iwf}\arrow[ur] \arrow[dr]&&& \abs{\Iwf}\arrow[r]&2^{\abs{X}}\\
&&\non{\Iwf}\arrow[r]\arrow[uur]&\abs{X}\arrow[ur]
\end{tikzcd}
$$
\caption{Inequalities between cardinal characteristics associated with $\Iwf$, $a\rightarrow b$ means $a\leq b$.}
\label{basicineq}
\end{figure}
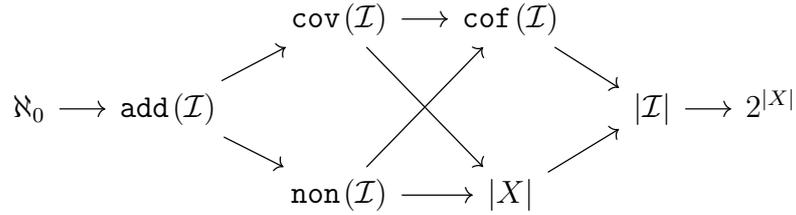

An ideal $\Iwf$ can be interpreted as a~partially ordered set $(\Iwf,{\subseteq})$.
We will consider also partial preorders~${\subseteq^\Jwf}$ on~$\Iwf$ for various ideals~$\Jwf$ on~$X$.
Since any preorder is a~relational system, we shall employ the theory of relational systems (see e.g., \cite{blass, vojtas}) to analyze partial preorders.

By a~\textit{relational system} we understand an ordered triple
$\Rbf=(X,Y,{\sqsubset})$ where $X$ and $Y$ are nonempty sets and ${\sqsubset}$ is a~subset of $X\times Y$.
Given a~relational system $\Rbf=(X,Y,{\sqsubset})$, a~set $A\subseteq X$ is said to be $\Rbf$-\textit{unbounded} if
$\forall y\in Y$ $\exists a\in A$ $\neg (a\sqsubset y)$; a~set
$B\subseteq Y$ is said to be $\Rbf$-\textit{cofinal} if
$\forall x\in X$ $\exists b\in B$ $x\sqsubset b$.
Define
\begin{align*}
&\bfrak(\Rbf)=\min\set{\abs{A}}{A\subseteq X
\text{ is }\Rbf\text{-unbounded}},\\
&\dfrak(\Rbf)=\min\set{\abs{A}}{A\subseteq X
\text{ is } \Rbf\text{-cofinal}}.
\end{align*}

The cardinal invariants $\bfrak(\Rbf)$ and $\dfrak(\Rbf)$ may be undefined.
In fact, $\bfrak(\Rbf)$ is undefined if and only if $\dfrak(\Rbf)=1$ and vice versa.
If $\bfrak(\Rbf)$ is not defined we write $\bfrak(\Rbf)=\infty$;
otherwise we write $\bfrak(\Rbf)<\infty$.
Similarly for $\dfrak(\Rbf)$.
The situation slightly changes when $\Rbf$~is a~partial preorder:
Due to reflexivity of~$\Rbf$ the cardinal value $\dfrak(\Rbf)$ is always defined.

A~partially preordered set $(X,{\sqsubset})$ can be thought of as a~relational system $(X,X,{\sqsubset})$, thus, we shall write $\bfrak(X,{\sqsubset})$ and $\dfrak(X,{\sqsubset})$ instead of $\bfrak(X,Y,{\sqsubset})$ and $\dfrak(X,Y,\sqsubset)$, respectively.
Notice that $\add{\Iwf}=\bfrak(\Iwf,{\subseteq})$ and $\cof{\Iwf}=\dfrak(\Iwf,{\subseteq})$.
Similarly, for ideals $\Iwf$ and $\Jwf$ on~$X$ let us define
\begin{align*}
&\addJ{\Jwf}{\Iwf}=\bfrak(\Iwf,{\subseteq^\Jwf})=\min
\set{\abs{\Awf}}
{\Awf\subseteq\Iwf\ \land\
\forall I\in\Iwf\ \exists A\in\Awf\ I\nsupseteq^\Jwf A},\\
&\cofJ{\Jwf}{\Iwf}=\dfrak(\Iwf,{\subseteq^\Jwf})=\min
\set{\abs{\Awf}}
{\Awf\subseteq\Iwf\ \land\
\forall I\in\Iwf\ \exists A\in\Awf\ I\subseteq^\Jwf A}.
\end{align*}

If $\Rbf=(X,Y,{\sqsubset})$ and $\Rbf'=(X',Y',{\sqsubset'})$ are relational systems, a pair of mappings $(\Psi_-,\Psi_+)$ is called a~\textit{Tukey connection from} $\Rbf$ \textit{to} $\Rbf'$ if
\begin{align*}
&\Psi_-\colon X\to X',
\quad
\Psi_+\colon Y'\to Y,
\quad\text{and}\\
&\forall x\in X\ \forall y'\in Y'\
(\Psi_-(x)\sqsubset' y'\Rightarrow x\sqsubset\Psi_+(y')).
\end{align*}
If there exists a Tukey connection from $\Rbf$
to $\Rbf'$, we write
$\Rbf\leqT\Rbf'$.
If $\Rbf\leqT \Rbf'$ and
$\Rbf\leqT\Rbf'$, we write
$\Rbf\eqT\Rbf'$.
One can show that if $\Rbf\leqT\Rbf'$,
then $\bfrak(\Rbf')\leq \bfrak(\Rbf)$ and
$\dfrak(\Rbf)\leq\dfrak(\Rbf')$.
Consequently, if $\Rbf\eqT\Rbf'$
then $\bfrak(\Rbf)= \bfrak(\Rbf')$ and
$\dfrak(\Rbf)=\dfrak(\Rbf')$.

We define a product of relational systems $\Rbf=(X,Y,{\sqsubset})$ and $\Rbf'=(X',Y',{\sqsubset'})$ by
\[
\Rbf\times\Rbf'=
(X\times X', Y\times Y', {\sqsubset_\times}),
\]
where $(x,x')\sqsubset_\times (y,y')$ iff $x\sqsubset x'$ and $y\sqsubset' y'$.

\begin{lemma}[\cite{blass}, Theorem~4.11]\label{prod}
\startlist
\begin{enumerate}[label=\normalfont(\alph*)]
\item
$\bfrak(\Rbf\times\Rbf')=
\min\{\bfrak(\Rbf), \bfrak(\Rbf')\}$,
\item
$\max\{\dfrak(\Rbf),\dfrak(\Rbf')\}\leq
\dfrak(\Rbf\times\Rbf')\leq
\dfrak(\Rbf)\cdot\dfrak(\Rbf')$.
\qed
\end{enumerate}
\end{lemma}

\begin{definition}
For two ideals $\Iwf$ and $\Jwf$ we define
\[
\Iwf\oplus\Jwf=\set{(I\times\{0\})\cup(J\times\{1\})}{I\in\Iwf\ \land\ J\in\Jwf}.
\]
\end{definition}

One can easily observe the following:

\begin{observation}\label{disjSumProp}
\startlist
\begin{enumerate}[label=\normalfont(\alph*)]
\item\label{disjSumPropc}
$\add{\Iwf\oplus\Jwf}=\min\{\add{\Iwf},\add{\Jwf}\}$.
\item\label{disjSumPropb}
$\cov{\Iwf\oplus\Jwf}=\max\{\cov{\Iwf},\cov{\Jwf}\}$.
\item\label{disjSumPropd}
$\non{\Iwf\oplus\Jwf}=\min\{\non{\Iwf},\non{\Jwf}\}$.
\item\label{disjSumPropa}
$\cof{\Iwf\oplus\Jwf}=\max\{\cof{\Iwf},\cof{\Jwf}\}$.
\qed
\end{enumerate}
\end{observation}

Note that
$(\Iwf\oplus\Jwf,{\subseteq})$ is isomorphic to
$(\Iwf,{\subseteq})\times(\Jwf,{\subseteq})$.
Therefore,
$(\Iwf\oplus\Jwf,{\subseteq})\eqT
(\Iwf,{\subseteq})\times(\Jwf,{\subseteq})$
and \ref{disjSumPropc} and \ref{disjSumPropa} can be considered as
instances of \autoref{prod}.

\section{Relative cofinality}\label{SecRelativeCofinality}

In this section we study basic properties of the relative cofinality.
We improve some results from \cite{tatra} and we present basic technical machinery.

Let us start by mentioning that we can consider relative cofinality not only as a generalization of $\cofJ{*}{\Iwf}$ from \cite{inv}, but also as a generalization of classical cofinality.
Usually, the set $\{\emptyset\}$ is considered ideal even if it does not meet the definition.
This approach is chosen purely for technical reasons, e.g., in the definition of Fubini product of ideals.
Concerning the relative cofinality, this approach is useful too, since clearly $\cofJ{\{\emptyset\}}{\Iwf}=\cof{\Iwf}$.

At this point we will focus on finding the bounds of the invariant of relative cofinality. Recall that in \cite{tatra}, the first author proved the following.\footnote{In the original paper \cite{tatra}, the lemma is stated for $X=\omega$.
However, it is very easy to see that the proofs work for any infinite set $X$.}

\begin{lemma}[A. Marton]\label{basicPropcofJ}
Let $\Iwf$, $\Jwf$, $\Jwf'$ be ideals on $X$.
\begin{enumerate}[label=\normalfont(\alph*)]
\item\label{basicPropcofJa}
If $\Iwf\subseteq\Jwf$ then\/ $\cofJ{\Jwf}{\Iwf}=1$.
In particular, $\cofJ{\Iwf}{\Iwf}=1$.
\item\label{basicPropcofJb}
$\cofJ{\Jwf}{\Iwf}\leq\cof{\Iwf}$.
\item\label{basicPropcofJc}
If $\Jwf\subseteq\Jwf'$ then\/ $\cofJ{\Jwf'}{\Iwf}\leq\cofJ{\Jwf}{\Iwf}$.
\item\label{basicPropcofJd}
Either\/ $\cofJ{\Jwf}{\Iwf}=1$ or\/ $\cofJ{\Jwf}{\Iwf}\geq\omega$.
\item\label{basicPropcofJe}
If $\Iwf$ is a\/ $\pid{\Jwf}$-ideal, then either\/ $\cofJ{\Jwf}{\Iwf}=1$ or\/ $\cofJ{\Jwf}{\Iwf}\geq\omega_1$.
\qed
\end{enumerate}
\end{lemma}

One can easily show that the dichotomy part \autoref{basicPropcofJ}\ref{basicPropcofJd} can be improved as follows.

\begin{lemma}\label{add<=cof}
If\/ $\cofJ{\Jwf}{\Iwf}>1$, then
$(\Iwf,{\nsupseteq})\leqT
(\Iwf,{\nsupseteq^\Jwf})\leqT
(\Iwf,\subseteq^\Jwf)\leqT
(\Iwf,{\subseteq})$.
As a~consequence, if\/ $\cofJ{\Jwf}{\Iwf}>1$, then
$\omega\leq\add{\Iwf}\leq\addJ{\Jwf}{\Iwf}\leq
\cofJ{\Jwf}{\Iwf}\leq\cof{\Iwf}$.
\end{lemma}

\begin{proof}
We describe the Tukey connection
$(\Psi_-,\Psi_+)\colon(\Iwf,\nsupseteq^\Jwf)\to(\Iwf,\subseteq^\Jwf)$; the other two connections are witnessed by identity functions.
Let $A,B\in\Iwf$.
Define $\Psi_+(B)=B$.
Since $\cofJ{\Jwf}{\Iwf}>1$ we can find $\Psi_-(A)\in\Iwf$ so that
$A\subseteq\Psi_-(A)$ and $\Psi_-(A)\nsubseteq^\Jwf A$.
Obviously,
$\Psi_-(A)\subseteq^\Jwf B\Rightarrow A\nsupseteq^\Jwf\Psi_+(B)$.
To prove the inequality
$\addJ{\Jwf}{\Iwf}\leq\cofJ{\Jwf}{\Iwf}$ apply the fact that
$\addJ{\Jwf}{\Iwf}=\dfrak(\Iwf,{\nsupseteq^\Jwf})$ and
$\cofJ{\Jwf}{\Iwf}=\dfrak(\Iwf,{\subseteq^\Jwf})$.
\end{proof}

We immediately get the following consequence.

\begin{corollary}\label{improvedDich}
\begin{flalign*}
&&&\cofJ{\Jwf}{\Iwf}>1\Leftrightarrow
\cofJ{\Jwf}{\Iwf}\geq\add{\Iwf}\Leftrightarrow
\addJ{\Jwf}{\Iwf}\leq\cof{\Iwf}\Leftrightarrow
\addJ{\Jwf}{\Iwf}<\infty.
&\llap{\qedsymbol}
\end{flalign*}
\end{corollary}

Note that
$\bfrak(X,{\sqsubset})\leq\cf(\dfrak(X,{\sqsubset}))$ for every
partially ordered set $(X,{\sqsubset})$ with infinite invariants
$\bfrak(\cdot)$ and $\dfrak(\cdot)$.
Therefore, if $\cofJ{\Jwf}{\Iwf}>1$, then by \autoref{add<=cof}
we have the following inequalities:
\[
\begin{tikzcd}
\add{\Iwf}\arrow[r]\arrow[drr]&\addJ{\Jwf}{\Iwf}\arrow[r]&\cf(\cofJ{\Jwf}{\Iwf})\arrow[r]&\cofJ{\Jwf}{\Iwf}\arrow[r]&\cof{\Iwf}\\
&&\cf(\cof{\Iwf})\arrow[rru]
\end{tikzcd}
\]

There are yet another two-parameter invariants $\add{\Iwf,\Jwf}$
and $\cof{\Iwf,\Jwf}$ (\cite[Definition 2.1.3]{bartjud})
for ideals $\Iwf\subseteq\Jwf$ defined by
\begin{align*}
&\add{\Iwf,\Jwf}=\bfrak(\Iwf,\Jwf,{\subseteq}),\\
&\cof{\Iwf,\Jwf}=\dfrak(\Iwf,\Jwf,{\subseteq}).
\end{align*}
Using the identity function defined on $\Iwf$ and the
$\Iwf$-coordinate of a~decomposition function from $\Iwf\vee\Jwf$
to $\Iwf\times\Jwf$ one can verify that
\[
(\Iwf,{\subseteq^\Jwf})\leqT(\Iwf,\Iwf\vee\Jwf,{\subseteq})\leqT(\Iwf,{\subseteq})
\]
and therefore
\[
\add{\Iwf}\leq\add{\Iwf,\Iwf\vee\Jwf}\leq\addJ{\Jwf}{\Iwf}\leq
\cofJ{\Jwf}{\Iwf}\leq\cof{\Iwf,\Iwf\vee\Jwf}\leq\cof{\Iwf}.
\]

We now find the characterization of the trivial case in \autoref{basicPropcofJ}\ref{basicPropcofJd}.
For this purpose we use the so-called restrictive inclusion.

\begin{definition}[\cite{my}]\label{restrictiveInc}
Let $\Iwf$ and $\Jwf$ be ideals on $X$.
We write $\Iwf\rsub\Jwf$ iff there is $E\in\Iwf^*$ such that $\Iwf{\restriction}E\subseteq\Jwf$.
\end{definition}

One can easily see that $\Iwf\rsub\Jwf$ if and only if $\Iwf$ has a $\Jwf$-union in $\Iwf$.
Recall that a set $E\subseteq X$ is $\Jwf$-union of $\Iwf$ if $I\subseteq^\Jwf E$ for each $I\in\Iwf$.

\begin{observation}\label{orthimpliescofJ1}
Let $\Iwf$ and $\Jwf$ be ideals on $X$.
\begin{enumerate}[label=\normalfont(\alph*)]
\item\label{orthimpliescofJ1a}
The following are equivalent.
\begin{enumerate}[label=\normalfont(\arabic*)]
\item\label{cofJI12}
$\cofJ{\Jwf}{\Iwf}=1$.
\item\label{cofJI11}
$\Iwf\rsub\Jwf$.
\item\label{cofJI13}
There is $A\in\Iwf$ such that $\Iwf\subseteq\gen{A}\vee\Jwf$.
\end{enumerate}
\item\label{orthimpliescofJ1b}
If $\Iwf\perp\Jwf$, then\/ $\cofJ{\Jwf}{\Iwf}=\cofJ{\Iwf}{\Jwf}=1$.
\end{enumerate}
\end{observation}

\begin{proof}
\ref{orthimpliescofJ1a}:
\ref{cofJI11}${}\Rightarrow{}$\ref{cofJI12} By the assumption there is $A\in\Iwf$ such that $I\subseteq^\Jwf A$ for all $I\in\Iwf$.
Thus, $\{A\}$~is clearly a witness for $\cofJ{\Jwf}{\Iwf}$.

\ref{cofJI12}${}\Rightarrow{}$\ref{cofJI13} Let $\{A\}$ be a witness for $\cofJ{\Jwf}{\Iwf}$.
Express any $I\in\Iwf$ as $I=(I\cap A)\cup (I\setminus A)$.
Then $I\cap A\subseteq A$ and $I\setminus A\in\Jwf$, thus, $I\in\gen{A}\vee\Jwf$.

\ref{cofJI13}${}\Rightarrow{}$\ref{cofJI11} $X\setminus A$ is clearly a witness for $\Iwf\rsub\Jwf$.

\ref{orthimpliescofJ1b}: Since $\Iwf\perp\Jwf$ implies both, $\Iwf\rsub\Jwf$ and $\Jwf\rsub\Iwf$, the result follows immediately from~\ref{orthimpliescofJ1a}.
\end{proof}

The following simple observation states that it is enough to consider pairs $\Iwf$, $\Jwf$ such that $\Jwf$ is a subideal of $\Iwf$.
Next, we show an easy way how to construct bases of $\Iwf$ from witnesses of the relative cofinality.

\begin{observation}\label{basecompletion}
Let $\Iwf$, $\Jwf$ be ideals on $X$.
\begin{enumerate}[label=\normalfont(\alph*)]
\item\label{intersectionIsEnough}
$\cofJ{\Jwf}{\Iwf}=\cofJ{\Jwf\cap\Iwf}{\Iwf}$.
\item\label{basecompletionb}
Let $\Jwf\subseteq\Iwf$, let $\Awf\subseteq\Iwf$ be cofinal
in\/ $(\Iwf,{\subseteq^\Jwf})$ and let $\Bwf_\Jwf$ be a base
of~$\Jwf$.
Then $\Awf'=\set{A\cup J}{A\in\Awf\ \land\ J\in\Bwf_\Jwf}$
is a~base of~$\Iwf$.
\item\label{basecompletionc}
If $\Jwf\subseteq\Iwf$, then\/
$\cof{\Iwf}\leq\cofJ{\Jwf}{\Iwf}\cdot\cof{\Jwf}$.
\end{enumerate}
\end{observation}

\begin{proof}
Item \ref{intersectionIsEnough} is trivial and
\ref{basecompletionc} is a~consequence of \ref{basecompletionb}.
We prove \ref{basecompletionb}.

Let $I\in\Iwf$.
By the assumptions there is $A\in\Awf$ such that
$I\setminus A\in\Jwf$.
Take $J\in\Bwf_\Jwf$ such that $I\setminus A\subseteq J$.
Then $I\subseteq(I\setminus A)\cup A\subseteq J\cup A$ and
$J\cup A\in\Awf'$.
\end{proof}

\begin{prop}\label{RCdepOnCof}
Let $\Iwf$ and $\Jwf$ be ideals on an infinite set~$X$.
\begin{enumerate}[label=\normalfont(\arabic*)]
\item\label{RCdepOnCof1}
If $\Jwf\subseteq\Iwf$ and\/ $\cof{\Jwf}\leq\cof{\Iwf}$,
then\/ $\cof{\Iwf}=\cofJ{\Jwf}{\Iwf}\cdot\cof{\Jwf}$.
\item\label{RCdepOnCof2}
If $\Jwf\subseteq\Iwf$ and\/ $\cof{\Jwf}<\cof{\Iwf}$,
then\/ $\cofJ{\Jwf}{\Iwf}=\cof{\Iwf}$.
\end{enumerate}
\end{prop}

\begin{proof}
\ref{RCdepOnCof1}:
By \autoref{basecompletion}\ref{basecompletionc} and \autoref{basicPropcofJ}\ref{basicPropcofJb}
and because $\cof{\Iwf}$~is infinite we have
$\cof{\Iwf}\leq\cofJ{\Jwf}{\Iwf}\cdot\cof{\Jwf}\leq\cofJ{\Jwf}{\Iwf}\cdot\cof{\Iwf}=\cof{\Iwf}$.

\ref{RCdepOnCof2} is a~consequence of~\ref{RCdepOnCof1}.
\end{proof}

In the next two lemmas, we calculate the relative cofinality for ideals expressed by the operations $\cap$, $\vee$, and~$\oplus$.

\begin{lemma}\label{product}
Let $\Iwf$, $\Jwf$, $\Iwf_0$, $\Iwf_1$, $\Jwf_0$, $\Jwf_1$
be ideals on $X$.
\begin{enumerate}[label=\normalfont(\alph*)]
\item\label{product-a}
$(\Iwf,{\subseteq^{\Jwf_0\cap\Jwf_1}})\eqT
(\Iwf,{\subseteq^{\Jwf_0}})\times(\Iwf,{\subseteq^{\Jwf_1}})$.
As a~consequence,
\begin{align*}
&\addJ{\Jwf_0\cap\Jwf_1}{\Iwf}=
\min\{\addJ{\Jwf_0}{\Iwf},\addJ{\Jwf_1}{\Iwf}\},\\
&\cofJ{\Jwf_0\cap\Jwf_1}{\Iwf}=
\max\{\cofJ{\Jwf_0}{\Iwf},\cofJ{\Jwf_1}{\Iwf}\}.
\end{align*}

\item\label{product-b}
If $\Iwf_0\perp\Iwf_1$, then
$(\Iwf_0\cap\Iwf_1,{\subseteq^\Jwf})\eqT
(\Iwf_0,{\subseteq^\Jwf})\times(\Iwf_1,{\subseteq^\Jwf})$.
As a~consequence,
\begin{align*}
&\addJ{\Jwf}{\Iwf_0\cap\Iwf_1}=
\min\{\addJ{\Jwf}{\Iwf_0},\addJ{\Jwf}{\Iwf_1}\},\\
&\cofJ{\Jwf}{\Iwf_0\cap\Iwf_1}=
\max\{\cofJ{\Jwf}{\Iwf_0},\cofJ{\Jwf}{\Iwf_1}\}.
\end{align*}

\item\label{product-c}
If $\Iwf_0\not\perp\Iwf_1$ and $\Iwf_0\cap\Iwf_1\subseteq\Jwf$, then
$(\Iwf_0\vee\Iwf_1,{\subseteq^\Jwf})\eqT
(\Iwf_0,{\subseteq^\Jwf})\times(\Iwf_1,{\subseteq^\Jwf})$.
As a~consequence,
\begin{align*}
&\addJ{\Jwf}{\Iwf_0\vee\Iwf_1}=
\min\{\addJ{\Jwf}{\Iwf_0},\addJ{\Jwf}{\Iwf_1}\},\\
&\cofJ{\Jwf}{\Iwf_0\vee\Iwf_1}=
\max\{\cofJ{\Jwf}{\Iwf_0},\cofJ{\Jwf}{\Iwf_1}\}.
\end{align*}
\end{enumerate}
\end{lemma}

\begin{proof}
\ref{product-a}:
The following functions form Tukey connections:
\begin{align*}
&((\id_\Iwf,\id_\Iwf{}),{\cup})\colon
(\Iwf,{\subseteq^{\Jwf_0\cap\Jwf_1}})\to
(\Iwf,{\subseteq^{\Jwf_0}})\times(\Iwf,{\subseteq^{\Jwf_1}}),\\
&({\cup},(\id_\Iwf,\id_\Iwf))\colon
(\Iwf,{\subseteq^{\Jwf_0}})\times(\Iwf,{\subseteq^{\Jwf_1}})\to
(\Iwf,{\subseteq^{\Jwf_0\cap\Jwf_1}}),
\end{align*}
i.e., for every $A,B_0,B_1\in\Iwf$,
\begin{align*}
&(A\subseteq^{\Jwf_0}B_0\ \land\ A\subseteq^{\Jwf_1}B_1)\Rightarrow A\subseteq^{\Jwf_0\cap\Jwf_1}B_0\cup B_1\quad\text{and}\\
&B_0\cup B_1\subseteq^{\Jwf_0\cap\Jwf_1}A\Rightarrow
(B_0\subseteq^{\Jwf_0}A\ \land\ B_1\subseteq^{\Jwf_1}A).
\end{align*}

\ref{product-b}:
Let $I_0\in\Iwf_0$ be such that $I_1=X\setminus I_0\in\Iwf_1$.
We define Tukey connections
\begin{align*}
&(\Psi_-,\Psi_+)\colon
(\Iwf_0\cap\Iwf_1,{\subseteq^\Jwf})\to
(\Iwf_0,{\subseteq^\Jwf})\times(\Iwf_1,{\subseteq^\Jwf}),\\
&(\Phi_-,\Phi_+)\colon
(\Iwf_0,{\subseteq^\Jwf})\times(\Iwf_1,{\subseteq^\Jwf})\to
(\Iwf_0\cap\Iwf_1,{\subseteq^\Jwf}),
\end{align*}
where $\Psi_-=(\Psi_-^0,\Psi_-^1)$, $\Phi_+=(\Phi_+^0,\Phi_+^1)$ and for every
$A\in\Iwf_0\cap\Iwf_1$ and $(B_0,B_1)\in\Iwf_0\times\Iwf_1$,
we define
\begin{align*}
&\Psi_-^0(A)=\Psi_-^1(A)=A,&
&\Phi_-(B_0,B_1)=(B_0\cap I_1)\cup(B_1\cap I_0),\\
&\Psi_+(B_0,B_1)=B_0\cap B_1,&
&\Phi_+^i(A)=A\cup I_i,\quad i=0,\ 1.
\end{align*}
It is easy to verify that
\begin{align*}
&(\Psi_-^0(A)\subseteq^\Jwf B_0\ \land\
\Psi_-^1(A)\subseteq^\Jwf B_1)\Rightarrow
A\subseteq^\Jwf\Psi_+(B_0,B_1)
\quad\text{and}\\
&\Phi_-(B_0,B_1)\subseteq^\Jwf A\Rightarrow
(B_0\subseteq^\Jwf\Phi_+^0(A)\ \land\ B_1\subseteq^\Jwf\Phi_+^1(A)).
\end{align*}
To see the latter implication assume that
$(B_0\cap I_1)\cup(B_1\cap I_0)\subseteq^\Jwf A$.
Then we have $(B_0\setminus\Phi_+^0(A))\cup(B_1\setminus\Phi_+^1(A))=
(B_0\cap I_1\setminus A)\cup(B_1\cap I_0\setminus A)\in\Jwf$.

\ref{product-c}:
We define Tukey connections
\begin{align*}
&(\Psi_-,\Psi_+)\colon(\Iwf_0\vee\Iwf_1,{\subseteq^\Jwf})\to
(\Iwf_0,{\subseteq^\Jwf})\times(\Iwf_1,{\subseteq^\Jwf}),\\
&(\Psi_+,\Psi_-)\colon
(\Iwf_0,{\subseteq^\Jwf})\times(\Iwf_1,{\subseteq^\Jwf})\to
(\Iwf_0\vee\Iwf_1,{\subseteq^\Jwf}),
\end{align*}
where
$\Psi_-=(\Psi_-^0,\Psi_-^1)\colon\Iwf_0\vee\Iwf_1\to\Iwf_0\times\Iwf_1$
and for every
$A\in\Iwf_0\vee\Iwf_1$, $(B_0,B_1)\in\Iwf_0\times\Iwf_1$,
\[
A\subseteq\Psi_-^0(A)\cup\Psi_-^1(A)
\quad\text{and}\quad
\Psi_+(B_0,B_1)=B_0\cup B_1.
\]
Then
\begin{align*}
&(\Psi_-^0(A)\subseteq^\Jwf B_0\ \land\
\Psi_-^1(A)\subseteq^\Jwf B_1)\Rightarrow
A\subseteq^\Jwf\Psi_+(B_0,B_1)
\quad\text{and}\\
&\Psi_+(B_0,B_1)\subseteq^\Jwf A\Rightarrow
(B_0\subseteq^\Jwf\Psi_-^0(A)\ \land\ B_1\subseteq^\Jwf\Psi_-^1(A)).
\end{align*}

If $\Psi_-^i(A)\subseteq^\Jwf B_i$ for $i=0$,~$1$, then
$A\setminus(B_0\cup B_1)\subseteq
(\Psi_-^0(A)\setminus B_0)\cup(\Psi_-^1(A)\setminus B_1)\in\Jwf$.

If $B_0\cup B_1\subseteq^\Jwf A$, then $B_i\setminus A\in\Jwf$ and
\begin{flalign*}
&&&
B_i\setminus\Psi_-^i(A)\subseteq
\underbrace{(B_i\setminus A)}_{\substack{\in\Jwf}}\cup
\underbrace{(B_i\cap\Psi_-^{1-i}(A))}_{\substack
{\in\Iwf_0\cap\Iwf_1\subseteq\Jwf}},\quad i=0,\ 1.
&\llap{\qedsymbol}
\end{flalign*}
\let\qed\relax
\end{proof}

\begin{lemma}\label{disSum}
Let $\Iwf_X$, $\Jwf_X$ and $\Iwf_Y$, $\Jwf_Y$ be ideals on $X$ and $Y$, respectively.
Then
\[
(\Iwf_X\oplus\Iwf_Y,{\subseteq^{\Jwf_X\oplus\Jwf_Y}})\eqT
(\Iwf_X,{\subseteq^{\Jwf_X}})\times(\Iwf_Y,{\subseteq^{\Jwf_Y}}).
\]
As a~consequence,
\begin{align*}
&\addJ{\Jwf_X\oplus\Jwf_Y}{\Iwf_X\oplus\Iwf_Y}=
\min\{\addJ{\Jwf_X}{\Iwf_X},\addJ{\Jwf_Y}{\Iwf_Y}\},\\
&\cofJ{\Jwf_X\oplus\Jwf_Y}{\Iwf_X\oplus\Iwf_Y}=
\max\{\cofJ{\Jwf_X}{\Iwf_X},\cofJ{\Jwf_Y}{\Iwf_Y}\}.
\end{align*}
\end{lemma}

\begin{proof}
The proof is similar to that of \autoref{product}~\ref{product-c}.
\end{proof}

Now we are ready to show that the diagram \autoref{incRelCof} is complete, i.e., no arrow can be added in general.

\begin{theorem}\label{diagCountEx}
\startlist
\begin{enumerate}[label=\normalfont(\alph*)]
\item\label{diagCountExa}
There exist an infinite set $X$ and ideals $\Iwf$, $\Jwf$ on $X$ with\/ $\non{\Iwf}<\cofJ{\Jwf}{\Iwf}$.
\item\label{diagCountExb}
There exist an infinite set $X$ and ideals $\Iwf$, $\Jwf$ on $X$ with\/ $\cofJ{\Jwf}{\Iwf}<\non{\Iwf}$.
\item\label{diagCountExc}
There exist an infinite set $X$ and ideals $\Iwf$, $\Jwf$ on $X$ with\/ $\cov{\Iwf}<\cofJ{\Jwf}{\Iwf}$.
\item\label{diagCountExd}
There exist an infinite set $X$ and ideals $\Iwf$, $\Jwf$ on $X$ with\/ $\cofJ{\Jwf}{\Iwf}<\cov{\Iwf}$.
\end{enumerate}
\end{theorem}

\begin{proof}
\ref{diagCountExa}: Put $X=\omega$, $\Iwf=\ED$, and $\Jwf=\fin$.
Then $\cofJ{*}{\ED}=\cof{\ED}=\cfrak>\non{\ED}=\omega$.

\ref{diagCountExb}: Let $X=\R$ and let $\Iwf_0$, $\Iwf_1$ be such that $\cof{\Iwf_1}<\non{\Iwf_0}$, e.g., consider $\Iwf_0=[\R]^{<\cfrak}$ and let $\Iwf_1=\gen{\set{[a,b]}{a,b\in\R\ \land\ a\leq b}}$.
Then clearly $\non{\Iwf_0}=\cfrak$ and $\cof{\Iwf_1}=\omega$ since any finite union of bounded intervals is contained in some interval with rational endpoints.
Also note that $\Iwf_0\vee\Iwf_1$ exists (since complement of any element of $\Iwf_1$ has cardinality $\cfrak$) and $\Iwf_0\cap\Iwf_1\subseteq\Iwf_0$.
Then, applying \autoref{product}\ref{product-c} and \autoref{basicPropcofJ} we have
\begin{align*}
\cofJ{\Iwf_0}{\Iwf_0\vee\Iwf_1}&=\max\{\cofJ{\Iwf_0}{\Iwf_0},\cofJ{\Iwf_0}{\Iwf_1}\}=\max\{1,\cofJ{\Iwf_0}{\Iwf_1}\}=\cofJ{\Iwf_0}{\Iwf_1}\\
&\leq\cof{\Iwf_1}<\non{\Iwf_0}\leq\non{\Iwf_0\vee\Iwf_1}.
\end{align*}

\ref{diagCountExc}:
Take the same example as in \ref{diagCountExa}.

\ref{diagCountExd}:
Let $X=\R$ and let $\Iwf_0$, $\Iwf_1$ be such that $\cofJ{*}{\Iwf_1}=\cof{\Iwf_1}<\cov{\Iwf_0}$, e.g., $\Iwf_0=\fin$ and $\Iwf_1=\gen{\set{[a,b]}{a,b\in\R\ \land\ a\leq b}}$.
Then clearly $\cov{\Iwf_0}=\cfrak$.
Since $\cof{\Iwf_1}=\omega$, also $\cofJ{*}{\Iwf_1}=\omega$ because for any $J\in\Iwf_1$ there is an interval $[c,d]$, $c<d$ disjoint with~$J$, thus $[c,d]\setminus J$ is infinite and consequently $\cofJ{*}{\Iwf_1}\geq\omega$.
Then applying \autoref{disSum} and \autoref{disjSumProp}\ref{disjSumPropb} we have
\begin{flalign*}
&&\cofJ{\Iwf_0\oplus\fin}{\Iwf_0\oplus\Iwf_1}&=\cofJ{\fin\oplus\fin}{\fin\oplus\Iwf_1}=\max\{\cofJ{*}{\fin},\cofJ{*}{\Iwf_1}\}=\cofJ{*}{\Iwf_1}
\\
&&&<\cov{\Iwf_0}=
\max\{\cov{\Iwf_0},\cov{\Iwf_1}\}=\cov{\Iwf_0\oplus\Iwf_1}.
&&\llap{$\qed$}
\end{flalign*}
\let\qed\relax
\end{proof}

\section{Ideals on \texorpdfstring{$\omega$}{omega}}\label{SecIdealsOnOmega}

In this section we are going to study particular instances of ideals on $\omega$.
In fact, we are going to fill missing values of Table~2 in \cite{tatra}, i.e., we are going to prove the results marked with cyan color in \autoref{completeTablecofJ}.

First, let us recall the definitions of the studied ideals.
\[
\begin{array}{l@{{}={}}l}
\fin&[\omega\times\omega]^{<\omega},\\[2pt]
\fino&\set{I\subseteq\omega\times\omega}
{\forall^\infty n\in\omega\ \set{m}{(n,m)\in I}=\emptyset},\\[2pt]
\ofin&\set{I\subseteq\omega\times\omega}
{\forall n\in\omega\ \set{m}{(n,m)\in I}\text{ is finite}},\\[2pt]
\finfin&\set{I\subseteq\omega\times\omega}
{\forall^\infty n\in\omega\
\set{m}{(n,m)\in I}\text{ is finite}},\\[2pt]
\sel&\set{I\subseteq\omega\times\omega}
{\exists k\in\omega\
\forall n\in\omega\ \abs{\set{m}{(n,m)\in I}}<k},\\[2pt]
\ED&\set{I\subseteq\omega\times\omega}
{\exists k\in\omega\
\forall^\infty n\in\omega\ \abs{\set{m}{(n,m)\in I}}<k}\\[2pt]
&\set{I\subseteq\omega\times\omega}
{\limsup\limits_{n\to\infty}\abs{(\{n\}\times\omega)\cap I}<\infty}.
\end{array}
\]

We will need the following results due to Meza-Alc\'antara \cite{meza}.

\begin{theorem}[D. Meza-Alc\'antara]
\startlist
\begin{enumerate}[label=\normalfont(\alph*)]
\item
$\cofJ{*}{\sel}=\cofJ{*}{\ED}=\cfrak$.
\item
$\cofJ{*}{\ofin}=\cofJ{*}{\finfin}=\dfrak$.
\item
$\cofJ{*}{\fino}=\omega$.
\qed
\end{enumerate}
\end{theorem}
The last observation before we proceed with the main result of this section is the following.
Note that $\cofJ{\Jwf}{\Iwf}=1$ immediately implies that $\Iwf$ is a $\pid{\Jwf}$-ideal.
Hence we have

\begin{fact}\label{cofJPJ}
If $\Iwf$ is not a $\pid{\Jwf}$-ideal, then\/ $\cofJ{\Jwf}{\Iwf}\geq\add{\Iwf}\geq\omega$.
\qed
\end{fact}

Now we can fill the empty spaces of Table~2 in \cite{tatra} (cyan values in \autoref{completeTablecofJ}), i.e., the cases when $\Iwf$ is not a $\pid{\Jwf}$-ideal.

\begin{theorem}\label{tablecompletion}
\startlist
\begin{enumerate}[label=\normalfont(\alph*)]
\item\label{tca}
$\cofJ{\ofin}{\fino}=\cofJ{\sel}{\fino}=\omega$.
\item\label{tcb}
$\cofJ{\ofin}{\finfin}=\omega$ and\/ $\cofJ{\sel}{\finfin}=\dfrak$.
\item\label{tcc}
$\cofJ{\fino}{\sel}=\cfrak$.
\item\label{tcd}
$\cofJ{\fino}{\ED}=\cfrak$ and\/ $\cofJ{\ofin}{\ED}=\cofJ{\sel}{\ED}=\omega$.
\end{enumerate}
\end{theorem}

\begin{proof}
\ref{tca}: Since $\fino$ is neither a $\pid{\ofin}$-ideal nor a $\pid{\sel}$-ideal, it follows from the fact \autoref{cofJPJ} that both invariants are at least $\omega$.
By \autoref{basicPropcofJ}\ref{basicPropcofJb} we have that both are also at most $\omega$.

\ref{tcb}: Since $\finfin=(\fino)\vee(\ofin)$ and $(\fino)\cap(\ofin)=[\omega\times\omega]^{<\omega}\subseteq\ofin$, applying \autoref{product}\ref{product-c} we have
\[
\cofJ{\ofin}{\finfin}=\max\{\cofJ{\ofin}{\fino},\cofJ{\ofin}{\ofin}\}.
\]
It follows from \ref{tca} that this is equal to $\max\{\omega,1\}=\omega$.

$\cofJ{\sel}{\finfin}\leq\dfrak$ follows from \autoref{basicPropcofJ}\ref{basicPropcofJb}.
$\cofJ{\sel}{\finfin}\geq\dfrak$ follows from \autoref{basicPropcofJ}\ref{basicPropcofJc} because $\sel\subseteq\ED$ and by \autoref{completeTablecofJ} we have $\cofJ{\ED}{\finfin}=\dfrak$.

\ref{tcc}: By \autoref{basecompletion}\ref{intersectionIsEnough}, $\cofJ{\fino}{\sel}=\cofJ{(\fino)\cap\sel}{\sel}=\cofJ{*}{\sel}=\cfrak$.

\ref{tcd}: Since $\fino\subseteq\ED$ and $\omega=\cof{\fino}<\cof{\ED}=\cfrak$, applying \autoref{RCdepOnCof} yields the result.

$\cofJ{\ofin}{\ED}=\cofJ{\ED\cap(\ofin)}{\ED}=\cofJ{\sel}{\ED}$.
Again, using \autoref{product}\ref{product-c} we have
\[
\cofJ{\sel}{\ED}=\max\{\cofJ{\sel}{\fino},\cofJ{\sel}{\sel}\}.
\]
By \ref{tca} this is equal to $\max\{\omega,1\}=\omega$.
\end{proof}

\section{Ideals on \texorpdfstring{$[0,1]$}{[0,1]}}\label{SecIdealsOnUnit}

In this section we shift our investigation to an uncountable Polish space, namely, the unit interval $\unit$ endowed with measure and topology inherited from $\R$, i.e., the standard Lebesgue measure and the standard topology induced by Euclidean metric.

Recall that a metric space $(X,d)$ has \textit{strong measure zero} if for every sequence $\seq{\varepsilon_n}{n\in\omega}$ of positive reals there exists a sequence $\seq{X_n}{n\in\omega}$ such that for every $n$, $X_n$ has diameter~$<\varepsilon_n$, and $X\subseteq\bigcup_{n\in\omega}X_n$.

We will study the $\mathrm{P}(\Jwf)$ property and the relative cofinality for pairs of the following well-known ideals.
\begin{align*}
\Nwf&=\set{A\subseteq\unit}{\mu(A)=0},\\
\Mwf&=\set{A\subseteq\unit}{A \text{ is meager}},\\
\Ewf&=\set{A\subseteq\unit}{A \text{ is covered by some null }F_\sigma\text{ subset of }\unit},\\
\SN&=\set{A\subseteq\unit}{A\text{ has strong measure zero}},\\
\nwd&=\set{A\subseteq\unit}{A\text{ is nowhere dense}}.
\end{align*}

Since $\Nwf$, $\Mwf$, $\Ewf$ and $\SN$ are $\sigma$-ideals, they are also $\mathrm{P}$-ideals.
Thus, the only non-trivial case regarding the $\pid{\Jwf}$ property is the nowhere dense ideal.
By $\nwd\subseteq\Mwf$ we have that $\nwd$ is a $\pid{\Mwf}$-ideal.
Next, we are going to investigate the $\pid{\Jwf}$ property of $\nwd$ when $\Jwf=\Nwf$, $\Ewf$, $\SN$.
We will need the following facts.

\begin{fact}[folklore]\label{measureofnwd}
If $A\subseteq\unit$ is a measurable nowhere dense set, then $\mu(A)<1$.
Otherwise there would exist a nonempty open set of measure zero.
\qed
\end{fact}

\begin{lemma}\label{sigmaP}
Let $\Iwf$ be an ideal on $X$ and $\Jwf$ be a $\sigma$-ideal on $X$.
Then $\Iwf$ is a\/ $\pid{\Jwf}$-ideal if and only if for any\/ $\set{I_n}{n\in\omega}\subseteq\Iwf$ there is $I\in\Iwf$ such that\/
$\left(\bigcup_{n\in\omega}I_n\right)\setminus I\in\Jwf$.
\end{lemma}

\begin{proof}
The ``if'' part is clear.
The ``only if'' part is easy to prove as well.
Let $\set{I_n}{n\in\omega}$ be a sequence of sets in $\Iwf$.
By the assumption we have $I_n\setminus I\in\Jwf$ for each $n\in\omega$.
Since $\Jwf$ is a $\sigma$-ideal, $\bigcup_{n\in\omega}(I_n\setminus I)\in\Jwf$.
\end{proof}

\begin{theorem}\label{nwd}
$\nwd$ is not a\/ $\pid{\Nwf}$-ideal.
\end{theorem}

\begin{proof}
Let $\set{A_n}{n\in\omega}\subseteq\unit$ be a sequence of nowhere dense subsets of $\unit$ such that $\mu\left(\bigcup_{n\in\omega}A_n \right)=1$ and let wlog $A\in\nwd$ be closed (measurable).
We have $\mu(A)<1$ and therefore
\[
\mu\left(\left(\bigcup_{n\in\omega}A_n \right)\setminus A\right)\geq
1-\mu(A)>0.
\]
Thus, applying \autoref{sigmaP} we get the desired result.
\end{proof}

\begin{corollary}\label{nwdSN}
$\nwd$ is neither a\/ $\pid{\Ewf}$-ideal nor a\/ $\pid{\SN}$-ideal.
\qed
\end{corollary}

Next, we are going to establish a technical machinery consisting of known results about measure and category.
Afterwards, we will use these facts to evaluate the relative cofinality for each pair of the above-mentioned ideals. Recall that a nonempty closed set $C\subseteq \unit$ is \textit{$\Nwf$-perfect} if $U\cap C\neq\emptyset$ implies $\mu(U\cap C)>0$ for any open set $U\subseteq\unit$. An uncountable set $X\subseteq\unit$ is a~\textit{generalized Luzin set} if for every meager set~$F$, $\abs{F\cap X}<\abs{X}$.

\begin{fact}\label{FACTS}
\startlist
\begin{enumerate}[label=\normalfont(\alph*)]
\item\label{complementOfSMZ}
\emph{(folklore)}
$\SN\not\perp\Mwf$.
\item\label{NMareorth}
\emph{(folklore)}
$\Nwf\perp\Mwf$.
\item\label{measureFactsc}
\emph{(\cite[Lemma 1.1]{balcerzakSigmaIdeals})}
Let $F\subseteq\unit$ be an $\Nwf$-perfect set.
If $A\subseteq F$ is a $G_\delta$ set dense in $F$ then $A\notin \Ewf$.
\item\label{measureFactsd}
\emph{(\cite[Lemma 2]{balcerzakMauldin})}
Every set $A\subseteq\unit$ of type~$F_\sigma$ such that $A\notin \Nwf$ contains an $\Nwf$-perfect set nowhere dense in $A$.
\item\label{perfectSMZ1}
\emph{(folklore)}
No perfect set has strong measure zero.
\item\label{souslin}
\emph{(M.J. Souslin)}
An uncountable analytic subset of a Polish space contains
a perfect subset homeomorphic to the Cantor middle-third set.
\item\label{gLuzin}
\emph{(\cite[Lemma 8.2.4]{bartjud})}
Every generalized Luzin set of regular cardinality is nonmeager
and has strong measure zero.
\qed
\end{enumerate}
\end{fact}

Also, we shall need the following lemma.

\begin{lemma}\label{MRLemma}
For every $G_\delta$ dense set $G\subseteq\unit$ there is a nowhere dense $\Nwf$-perfect set $K\subseteq\unit$ such that $K\cap G$ is dense in $K$.
\end{lemma}

\begin{proof}
Let $G=\bigcap_{n\in\omega} U_n$, where $U_{n+1}\subseteq U_n\subseteq\unit$ are open dense sets (it may happen that $\mu(G)=0$).
Fix a~sequence of positive reals $r_n<1$ such that $\prod_{n\in\omega}r_n>0$.
By induction construct finite families $\Kwf_n$ of pairwise disjoint
nonempty closed intervals,
$\Kwf_{n+1}^I=\set{J\in\Kwf_{n+1}}{J\subseteq I}$ for $I\in\Kwf_n$,
and selectors $s_n\colon\Kwf_n\to G$ for $n\in\omega$ so that the conditions \ref{MRLemma1}--\ref{MRLemma4} are satisfied.
\begin{enumerate}[label=(\roman*)]
\item\label{MRLemma1}
$\Kwf_0=\{I_0\}$ for some $I_0\subseteq U_0$,
\item\label{MRLemma2}
$\forall I\in\Kwf_n$ $\exists J\in\Kwf_{n+1}^I$
$J\subseteq U_{n+1}$ and $s_n(I)=s_{n+1}(J)\in J\subseteq I$,
\item\label{MRLemma3}
$\forall I\in\Kwf_n$ $\abs{\Kwf_{n+1}^I}\geq2$ and
$\mu(\bigcup\Kwf_{n+1}^I)=r_n\cdot\mu(I)$,
\item\label{MRLemma4}
$\Kwf_{n+1}=\bigcup_{I\in\Kwf_n}\Kwf_{n+1}^I$ and
$\forall I\in\Kwf_n$ $\forall J\in\Kwf_{n+1}^I$ $\mu(J)<\mu(I)/2$.
\end{enumerate}
Let $K=\bigcap_{n\in\omega}\bigcup\Kwf_n$.
Clearly, the set $K$~is closed and by
\ref{MRLemma3}--\ref{MRLemma4}, $K$~is nowhere dense and the family
$\set{K\cap I}{I\in\bigcup_{n\in\omega}\Kwf_n}$ is a~base of
relatively clopen subsets of~$K$.
It follows by \ref{MRLemma2} that the set
$Q=\bigcup_{n\in\omega}\set{s_n(I)}{I\in\Kwf_n}$ is a subset of~$G$ as well as a~subset of~$K$.
The set~$Q$ is dense in~$K$ because it meets every clopen subset
of~$K$.
Therefore $K\cap G$ is dense in~$K$.
Finally, $K$ is $\Nwf$-perfect because by \ref{MRLemma3}, for every $n\in\omega$ and $I\in\Kwf_n$,
$\mu(K\cap I)=(\prod_{k>n}r_k)\cdot\mu(I)>0$.
\end{proof}

\begin{theorem}\label{nwdBaireTukey}
$(\baire,{\leq^*})\leqT(\nwd,{\subseteq^\Nwf})$.
Therefore\/
$\addJ{\Nwf}{\nwd}\leq\bfrak$ and\/
$\dfrak\leq\cofJ{\Nwf}{\nwd}$.
\end{theorem}

\begin{proof}
Let $\seq{I_n}{n\in\omega}$ be a~sequence of pairwise disjoint
open intervals in $[0,1]$ and for every $n\in\omega$ let
$\seq{C_{n,k}}{k\in\omega}$ be an increasing sequence of nowhere
dense subsets of~$I_n$ such that
$\mu(\bigcup_{m\in\omega}A_{n,m})=\mu(I_n)$.
Define
\[
\Psi_-\colon\baire\to\nwd
\quad\text{and}\quad
\Psi_+\colon\nwd\to\baire
\]
by $\Psi_-(f)=\bigcup_{n\in\omega}C_{n,f(n)}$ and
$\Psi_+(A)(n)=\min\set{k\in\omega}{\mu(A\cap I_n)<\mu(C_{n,k})}$ for $f\in\baire$ and $A\in\nwd$;
$\Psi_+(A)$~is well defined because $\mu(A\cap I_n)<\mu(I_n)$.
One can easily verify that
$\Psi_-(f)\subseteq^\Nwf A\Rightarrow\forall n$
$f(n)\leq\Psi_+(A)(n)$.
\end{proof}

Before we state the next theorem, notice that $\add{\nwd}=\aleph_0$, since it is not a $\sigma$-ideal~-- there exists a meager set of measure $1$ (see \autoref{FACTS}\ref{NMareorth}) and by \autoref{measureofnwd} this set cannot be nowhere dense.
Also, recall that any null set can be covered by some null set of type~$G_\delta$.
Finally, we will prove the basic dichotomy (trivial/non-trivial) for the studied ideals on uncountable spaces.

\begin{theorem}\label{uncountableCof}
\startlist
\begin{tasks}
[label-format={\normalfont},label={(\alph*)},label-align=right,
item-indent={38pt},label-width={20pt},label-offset={5pt},
after-item-skip={6pt}](2)
\task\label{nwdchar}%
$\cofJ{\Nwf}{\nwd}\geq\dfrak$,\\
$\cofJ{\Ewf}{\nwd}\geq\dfrak$,\\
$\cofJ{\Mwf}{\nwd}=1$,\\
$\cofJ{\SN}{\nwd}\geq\dfrak$.

\task\label{Nchar}%
$\cofJ{\nwd}{\Nwf}\geq\add{\Nwf}$,\\
$\cofJ{\Ewf}{\Nwf}\geq\add{\Nwf}$,\\
$\cofJ{\Mwf}{\Nwf}=1$,\\
$\cofJ{\SN}{\Nwf}\geq\add{\Nwf}$.

\task\label{Mchar}%
$\cofJ{\nwd}{\Mwf}\geq\add{\Mwf}$,\\
$\cofJ{\Nwf}{\Mwf}=1$,\\
$\cofJ{\Ewf}{\Mwf}\geq\add{\Mwf}$,\\
$\cofJ{\SN}{\Mwf}\geq\add{\Mwf}$.

\task\label{Echar}%
$\cofJ{\nwd}{\Ewf}\geq\add{\Ewf}$,\\
$\cofJ{\Nwf}{\Ewf}=1$,\\
$\cofJ{\Mwf}{\Ewf}=1$,\\
$\cofJ{\SN}{\Ewf}\geq\add{\Ewf}$.

\task*\label{SNchar}%
$\cofJ{\nwd}{\SN}\geq\add{\SN}$,\\
$\cofJ{\Nwf}{\SN}=1$,\\
$\cons{\cofJ{\Ewf}{\SN}=\cofJ{\Mwf}{\SN}=1}$,\\
$\add{\Mwf}=\cof{\Mwf}\Rightarrow
\cofJ{\Ewf}{\SN}\geq\cofJ{\Mwf}{\SN}\geq\add{\SN}$.
\end{tasks}
\end{theorem}

\begin{proof}
Note that
\ref{nwdchar}~$\cofJ{\Mwf}{\nwd}=1$ because $\nwd\subseteq\Mwf$;
\ref{Nchar}--\ref{Mchar}~~$\cofJ{\Mwf}{\Nwf}=\cofJ{\Nwf}{\Mwf}=1$
holds by \autoref{orthimpliescofJ1}~\ref{orthimpliescofJ1b}
because $\Nwf\perp\Mwf$;
\ref{Echar}~$\cofJ{\Nwf}{\Ewf}=\cofJ{\Mwf}{\Ewf}=1$ because $\Ewf\subseteq\Nwf\cap\Mwf$; and
\ref{SNchar}~$\cofJ{\Nwf}{\SN}=1$ because $\SN\subseteq\Nwf$.
Assuming Borel's conjecture
(see \cite{laver}) we have
$\SN=[[0,1]]^{\leq\omega}\subseteq\Ewf\subseteq\Mwf$ and by \autoref{basicPropcofJ},
$\cons{\cofJ{\Ewf}{\SN}=\cofJ{\Mwf}{\SN}=1}$.

\ref{nwdchar}:
The inequality $\cofJ{\Nwf}{\nwd}\geq\dfrak$ holds by \autoref{nwdBaireTukey}.
Since $\Ewf\cup\SN\subseteq\Nwf$ we have
$\cofJ{\Ewf}{\nwd}\geq\cofJ{\Nwf}{\nwd}$ and
$\cofJ{\SN}{\nwd}\geq\cofJ{\Nwf}{\nwd}$.

\ref{Nchar}--\ref{SNchar}:
By \autoref{add<=cof}, to prove the assertions of the form
$\cofJ{\Jwf}{\Iwf}\geq\add{\Iwf}$ it is enough to show that
$\cofJ{\Jwf}{\Iwf}>1$.

$\cofJ{\nwd}{\Iwf}>1$ for $\Iwf=\Nwf$, $\Mwf$, $\Ewf$, $\SN$:
Let $A\in\Iwf$ be arbitrary.
There is a~countable dense set $D\subseteq\unit\setminus A$.
Clearly $D\in\Iwf$ and $D\nsubseteq^\nwd A$ because
$D\setminus A=D\notin\nwd$.

$\cofJ{\SN}{\Iwf}>1$ for $\Iwf=\Nwf$, $\Mwf$, $\Ewf$:
Let $A\in\Iwf$ be arbitrary.
By \autoref{FACTS}\ref{souslin} there is a~perfect set $C\subseteq\unit\setminus A$
such that $\mu(C)=0$.
Note that $C$ is nowhere dense.
By \autoref{FACTS}\ref{perfectSMZ1}, we have that $C\notin\SN$.
Then $C\in\Iwf$ and $C\nsubseteq^\SN A$
because $C\setminus A=C\notin\SN$.

$\cofJ{\Ewf}{\Nwf}>1$:
Let $A\in\Nwf$ be arbitrary.
There is an $\Nwf$-perfect set~$N$ disjoint from~$A$
(see e.g., \autoref{FACTS}\ref{measureFactsd}).
Let $G$ be a~$G_\delta$~subset of~$N$ with $\mu(G)=0$.
By \autoref{FACTS}\ref{measureFactsc}, $G\notin\Ewf$.
Thus $G\in\Nwf$ and $G\nsubseteq^\Ewf A$ because
$G\setminus A=G\notin\Ewf$.

$\cofJ{\Ewf}{\Mwf}>1$:
Let $A\in\Mwf$.
There is a~$G_\delta$ set $G\subseteq\unit$ disjoint from~$A$.
By \autoref{MRLemma} there is a nowhere dense $\Nwf$-perfect set $K\subseteq\unit$ such that $K\cap G$ is dense in $K$.
Obviously, $K\in\Mwf$.
We claim that $K\cap G\notin\Ewf$, which proves that $K\not\subseteq^\Ewf A$ and hence, $\cofJ{\Ewf}{\Mwf}>1$.
Towards a~contradiction, assume that $K\cap G\subseteq\bigcup_{n\in\omega}E_n'$ where $E_n'$ are closed measure zero sets.
The sets $E_n=K\cap E_n'$ are closed sets of measure zero and $K\cap G=\bigcup_{n\in\omega}E_n\cap G$.
The $G_\delta$ set $K\cap G$ is a complete metric space and by the Baire category theorem there is a nonempty relatively open set $U'\subseteq K\cap G$ such that $U'\subseteq E_n\cap G$ for some $n\in\omega$.
Let $U\subseteq\unit$ be an open set such that $U'=K\cap G\cap U$.
Then $\emptyset\neq K\cap G\cap U=E_n\cap G\cap U$.
Since $K$ and $E_n$ are closed sets and $K\cap G\cap U$ is a dense subset of $K\cap U$,
\begin{align*}
K\cap U&=\set{x\in U}{\text{$x$ is a limit of a sequence of points
from $K\cap G\cap U$}}\\
&=\set{x\in U}{\text{$x$ is a limit of a sequence of points from
$E_n\cap G\cap U$}}\subseteq E_n\cap U.
\end{align*}
Then $\mu(E_n)\geq\mu(E_n\cap U)\geq\mu(K\cap U)>0$ because $K$ is $\Nwf$-perfect.
A~contradiction.

$\add{\Mwf}=\cof{\Mwf}\Rightarrow
\cofJ{\Ewf}{\SN}\geq\cofJ{\Mwf}{\SN}>1$:
$\cofJ{\Ewf}{\SN}\geq\cofJ{\Mwf}{\SN}$ because
$\Ewf\subseteq\Mwf$.
Let $A\in\SN$ be arbitrary.
Denote $\kappa=\cof{\Mwf}$ and let $\set{M_\xi}{\xi<\kappa}$
be a~base of~$\Mwf$.
By \autoref{FACTS}\ref{complementOfSMZ} the complement of~$A$
is nonmeager.
Choose $x_\xi\in(\unit\setminus A)\setminus\bigcup_{\eta<\xi}M_\eta$ for $\xi<\kappa$;
this is possible because $\add{\Mwf}=\kappa$.
It is clear that $B=\set{x_\xi}{\xi<\kappa}$ is a~generalized Luzin set of size~$\kappa$
disjoint from~$A$.
By \autoref{FACTS}\ref{gLuzin}, $B\in\SN\setminus\Mwf$ and $B\nsubseteq^\Mwf A$ because $B\setminus A=B\notin\Mwf$.
\end{proof}

Since Bartoszy\'nski and Shelah \cite{BS1992} showed that $\cof{\Ewf}=\cof{\Mwf}$, we have the following.

\begin{corollary}
$\cons{\cofJ{\Ewf}{\Nwf}=\cof{\Nwf}}$.
\qed
\end{corollary}

\section{Maximal ideals}\label{secMaxIdeals}

In this section, we show how relative cofinality behaves when maximal ideals are being considered.

\begin{lemma}\label{restrict}
Let $\Iwf$ and $\Jwf$ be ideals on~$X$.
If $A\in\Iwf^*\setminus\Jwf$, then\/
$\cof{\Iwf}=\cof{\Iwf{\restriction}A}$ and\/
$\cofJ{\Jwf}{\Iwf}=\cofJ{\Jwf{\restriction}A}{\Iwf{\restriction}A}$.
\qed
\end{lemma}

\begin{lemma}\label{maxA}
If $\Iwf$ is a~maximal ideal on~$X$, then\/ $\cof{\Iwf}>\non{\Iwf}$.
\end{lemma}

\begin{proof}
Denote $\nu=\non{\Iwf}$.
Without loss of generality we can assume that $\abs{X}=\nu$
since otherwise, by \autoref{restrict}, we can pass to the maximal ideal $\Iwf{\restriction}A$ for some set $A\in\Iwf^*$
of cardinality~$\nu$.
Let $\set{A_\alpha}{\alpha<\nu}\subseteq\Iwf$ be arbitrary.
Since $[X]^{<\nu}\subseteq\Iwf$, by induction on $\alpha<\nu$ we can
construct a~sequence $\seq{x^i_\alpha}{(\alpha,i)\in\nu\times2}$ of pairwise distinct members of~$X$ such that
$\{x^0_\alpha,x^1_\alpha\}\cap A_\alpha=\emptyset$.
Let $B^i=\set{x^i_\alpha}{\alpha<\nu}$ for $i=0$,~$1$.
By maximality of~$\Iwf$, $B^i\in\Iwf$ for some~$i$.
Then for every $\alpha<\kappa$,
$x^i_\alpha\in B^i\setminus A_\alpha\ne\emptyset$ which proves that
$\set{A_\alpha}{\alpha<\kappa}$ is not a~base of~$\Iwf$.
It follows that $\cof{\Iwf}>\nu$.
\end{proof}

The following statement concerns the value of $\cofJ{\Jwf}{\Iwf}$ if at least one of the ideals $\Iwf$ and $\Jwf$ is maximal.

\begin{prop}\label{maxB}
Let $\Iwf$, $\Jwf$ be ideals on $X$.
\begin{enumerate}[label=\normalfont(\alph*)]
\item\label{maxBa}
If $\Jwf$ is maximal, then $\cofJ{\Jwf}{\Iwf}=1$.
\item\label{maxBb}
If $\Iwf$ is maximal, then the following conditions are equivalent:
\begin{enumerate}[label=\normalfont(\roman*)]
\item\label{maxBb1}
$\cofJ{\Jwf}{\Iwf}=1$.
\item\label{maxBb2}
$\cofJ{\Jwf}{\Iwf}\leq\omega$.
\item\label{maxBb3}
$\Iwf\rsub\Jwf$.
\item\label{maxBb4}
$\Jwf\nsubseteq\Iwf$ or $\Iwf\rsub\Jwf$.
\item\label{maxBb5}
$\Jwf\nsubseteq\Iwf$ or $\Jwf\subseteq\Iwf\rsub\Jwf$.
\end{enumerate}
\end{enumerate}
\end{prop}

\begin{proof}
\ref{maxBa}: Either $\Iwf\subseteq\Jwf$ or $\Iwf\perp\Jwf$.
In both cases $\cofJ{\Jwf}{\Iwf}=1$.
Note that this is the limit case in \autoref{basicPropcofJ}\ref{basicPropcofJc}.

\ref{maxBb}:
By \autoref{orthimpliescofJ1}\ref{orthimpliescofJ1a},
$\text{\ref{maxBb1}}\Leftrightarrow\text{\ref{maxBb3}}$.
If $\Jwf\nsubseteq\Iwf$, then by maximality of~$\Iwf$, $\Jwf\perp\Iwf$ and then by definitions, $\Iwf\rsub\Jwf$.
It follows that
$\text{\ref{maxBb3}}\Leftrightarrow\text{\ref{maxBb4}}\Leftrightarrow\text{\ref{maxBb5}}$.
The implication $\text{\ref{maxBb1}}\Rightarrow\text{\ref{maxBb2}}$
is trivial.
We prove $\text{\ref{maxBb2}}\Rightarrow\text{\ref{maxBb3}}$.

Assume that there is a~countable set $\set{A_n}{n\in\omega}\subseteq\Iwf$ such that
$\forall A\in\Iwf$ $\exists n\in\omega$ $A\setminus A_n\in\Jwf$.
Towards a~contradiction assume that $\Iwf\nrsub\Jwf$.
Wlog we can assume that $A_n\subseteq A_{n+1}$
for all $n\in\omega$.
Since $\Iwf\nrsub\Jwf$, by definition (or by an equivalence in \autoref{orthimpliescofJ1}\ref{orthimpliescofJ1a}), we can inductively choose a~sequence of sets $\set{B_n}{n\in\omega}\subseteq\Iwf\setminus\Jwf$ such that $B_n\cap (A_n\cup\bigcup_{i<n}B_i)=\emptyset$.
Let $C_0=\bigcup_{n\in\omega}B_{2n}$ and $C_1=\bigcup_{n\in\omega}B_{2n+1}$.
Since $\Iwf$ is maximal, there is $i=0$,~$1$ such that $C_i\in\Iwf$.
For every $n\in\omega$, $C_i\setminus A_n\supseteq C_i\setminus A_{2n+i}\supseteq B_{2n+i}\notin\Jwf$.
A~contradiction.
\end{proof}

If $\Jwf=[X]^{<\kappa}$, then for maximal ideals~$\Iwf$ there is a~simple dichotomy:
$\cofJ{\Jwf}{\Iwf}$ is either~$1$ or~$\geq\omega_1$ depending on whether $\kappa>\non{\Iwf}$ or
$\kappa\leq\non{\Iwf}$.

\begin{prop}\label{maxC}
Let $\Iwf$ be a~maximal ideal on~$X$ and let $\nu=\non{\Iwf}$.
\begin{enumerate}[label=\normalfont(\alph*)]
\item\label{maxCa}
If $\omega\leq\kappa\leq\nu$, then\/
$\cofJ{[X]^{<\kappa}}{\Iwf}\geq\omega_1$.
\item\label{maxCb}
If $\nu<\kappa\leq\abs{X}$, then
$\cofJ{[X]^{<\kappa}}{\Iwf}=1$.
\item\label{maxCc}
If $\kappa\ge\omega$ and $\nu^{<\kappa}<\cof{\Iwf}$, then\/
$\cofJ{[X]^{<\kappa}}{\Iwf}=\cof{\Iwf}$.
\item\label{maxCd}
$\cofJ{*}{\Iwf}=\cof{\Iwf}>\nu$.
\end{enumerate}
\end{prop}

\begin{proof}
\ref{maxCa} and \ref{maxCb}:
A~consequence of \autoref{maxB}\ref{maxBb} because
$\Iwf\rsub[X]^{<\kappa}\Leftrightarrow\kappa>\nu$.

\ref{maxCc}:
Denote $\Jwf=[X]^{<\kappa}$.
Since $\Iwf$~is maximal, there is a~set $A\in\Iwf^*$ of
cardinality~$\nu$.
We apply \autoref{restrict} for the set~$A$.
Now, $\cof{\Jwf{\restriction}A}\leq\nu^{<\kappa}<\cof{\Iwf}=
\cof{\Iwf{\restriction}A}$,
$\Jwf{\restriction}A\subseteq\Iwf{\restriction}A$,
and by \autoref{RCdepOnCof}\ref{RCdepOnCof2},
$\cofJ{\Jwf}{\Iwf}=\cofJ{\Jwf{\restriction}A}{\Iwf{\restriction}A}=
\cof{\Iwf{\restriction}A}=\cof{\Iwf}$.

\ref{maxCd}:
By \autoref{maxA} this is a~particular case of \ref{maxCc} for $\kappa=\omega$.
\end{proof}

Note that $\kappa\leq\nu$ in \autoref{maxC}~\ref{maxCc}
because $\nu^{<\kappa}<\cof{\Iwf}\leq2^\nu$.

According to Posp\'{\i}\v{s}il's theorem the smallest cardinality of a~relative base of a~maximal ideal can reach the cardinality of the power set.

\begin{prop}[Posp\'{\i}\v{s}il \cite{pospisil}]\label{maxD}
There is a~maximal ideal~$\Iwf$ on~$X$ such that\/
$\cof{\Iwf}=\cofJ{*}{\Iwf}=2^{\abs{X}}$.
\end{prop}

\begin{proof}
The proof is in accordance with \cite[9.5~Proposition]{blass}.
Let $\Awf$ be an independent family of subsets of~$X$ of cardinality~$2^{\abs{X}}$, i.e., for any finitely many distinct members
$A_i,B_j\in\Awf$,
$\abs{X\setminus(\bigcup_iA_i\cup\bigcup_j(X\setminus B_j))}=\abs{X}$
(see e.g., \cite[Exercises for Chapter VIII, (A6)]{kunen}).
Let
$\Vwf=\Awf\cup
\set{X\setminus\bigcup\Wwf}{\Wwf\in[\Awf]^{\geq\omega}}$.
There is a~maximal ideal~$\Iwf$ containing~$\Vwf$.
By \autoref{maxC}\ref{maxCc} and \autoref{maxA}, $\cofJ{*}{\Iwf}=\cof{\Iwf}$.
Towards a~contradiction assume that $\cof{\Iwf}<2^{\abs{X}}$ and let
$\Cwf\subseteq\Iwf$ be a~cofinal subset of~$\Iwf$ of cardinality
$\abs{\Cwf}<2^{\abs{X}}$.
There is $C\in\Cwf$ such that the set
$\Wwf=\set{A\in\Awf}{A\subseteq C}$ is infinite.
Then $\bigcup\Wwf\in\Iwf$ because $\bigcup\Wwf\subseteq C\in\Iwf$
and by definition of~$\Vwf$ and $\Iwf$ also
$X\setminus\bigcup\Wwf\in\Iwf$.
This contradicts the maximality of~$\Iwf$.
\end{proof}

A~family $\Awf\subseteq\Pwf(X)$ is said to have
the \textit{finite union property}, abbreviated \textit{f.u.p.},
if $X\setminus\bigcup\Awf_0$ is infinite for every finite $\Awf_0\subseteq\Awf$.
A~family $\Awf\subseteq\Pwf(X)$ is said to have
the \textit{$\Jwf$-finite union property}, abbreviated
\textit{$\Jwf$-f.u.p.},
if $X\setminus\bigcup\Awf_0\in\Jwf^+$ for every finite $\Awf_0\subseteq\Awf$
(equivalently, if $\Awf\cup\Jwf$ has f.u.p.).
For an ideal~$\Jwf$ define
\begin{align*}
&\zfrak(\Jwf)=\min\set{\abs{\Awf}}{\Awf\subseteq\Jwf^+\ \land\
\text{$\Awf$~has $\Jwf$-f.u.p.}\ \land\
\forall B\in\Jwf^+\
\exists A\in\Awf\
A\cap B\in\Jwf^+},\\
&\zfrak'(\Jwf)=\min\set{\abs{\Awf}}{\Awf\subseteq\Jwf^+\ \land\
\text{$\Awf$~has $\Jwf$-f.u.p.}\ \land\
\forall B\in\Jwf^+\
\exists A\in\Awf\
\abs{A\cap B}\geq\omega}.
\end{align*}
Note that $\Awf\subseteq\Jwf^+\setminus\Jwf^*$ for~$\Awf$
in definitions of $\zfrak(\Jwf)$ and $\zfrak'(\Jwf)$.
In case $X=\omega$,
$\zfrak(\fin)=\zfrak'(\fin)=\pfrak$.
For comparison, let us consider one of the characterizations
of~$\pfrak$:
\[
\pfrak=\min\set{\abs{\Awf}}{\Awf\subseteq\Pwf(\omega)\ \land\
\text{$\Awf$~has f.u.p.}\ \land\
\forall B\in[\omega]^\omega\
\exists A\in\Awf\
A\nsubseteq^*\omega\setminus B}.
\]
Also note that if $\Awf\subseteq\Jwf^+$ is a~family witnessing the value $\zfrak(\Jwf)$,
then $\Iwf=\gen{\Jwf\cup\Awf}$ is an ideal, $\cofJ{\Jwf}{\Iwf}=\zfrak(\Jwf)$ and $\Iwf$~is $\Jwf$-tall,
i.e., for every $A\in\Iwf^+$ there is $B\subseteq A$ such that $B\in\Iwf\setminus\Jwf$.

\begin{lemma}\label{zfrak}
Let $\Jwf$ be an ideal on~$X$.
\begin{enumerate}[label=\normalfont(\alph*)]
\item\label{zfraka}
$\zfrak'(\Jwf)$ is defined if and only if $\Jwf$~is
a~non-maximal ideal.
\item\label{zfrakb}
$\zfrak(\Jwf)$ is defined if and only if there is a~maximal ideal $\Iwf\supseteq\Jwf$ such that $\Iwf\nrsub\Jwf$.
\item\label{zfrakc}
$\omega\leq\zfrak'(\Jwf)\leq\zfrak(\Jwf)\leq\abs{\Pwf(X)}$
whenever $\zfrak'(\Jwf)$ and $\zfrak(\Jwf)$ are defined.
\end{enumerate}
\end{lemma}

\begin{proof}
\ref{zfraka}
If $\Jwf$ is a~maximal ideal, then $\zfrak'(\Jwf)$ is not defined
because there is no nonempty family $\Awf\subseteq\Jwf^+$
with $\Jwf$-f.u.p..
Assume that an ideal $\Jwf$ is not maximal and let
$\Iwf\supseteq\Jwf$ be a~maximal ideal.
There is a~partition $B_0\cup B_1=X$ such that $B_0,B_1\in\Jwf^+$
and $B_0\in\Iwf$.
The set $\Awf=\Iwf\setminus\Jwf$ has $\Jwf$-f.u.p.\ and
we verify that $\Awf$ is a~witness for $\zfrak'(\Jwf)\leq\abs{\Awf}$.
Let $B\in\Jwf^+$.
If $B\in\Iwf$, then for $A=B$ we have $A\cap B\in\Jwf^+$ and so
$\abs{A\cap B}=\omega$.
If $B\in\Iwf^*$, then find $A_0\in\Iwf\cap[B]^\omega$
and consider $A=B_0\cup A_0$.
It is obvious that $A\in\Awf$ and $\abs{A\cap B}=\omega$.

\ref{zfrakb}
Assume that there is a~maximal ideal $\Iwf\supseteq\Jwf$
such that $\Iwf\nrsub\Jwf$.
The family $\Awf=\Iwf\setminus\Jwf$ has $\Jwf$-f.u.p.\
because $\Iwf\cap\Jwf^*=\emptyset$.
We verify that $\Awf$~witnesses $\zfrak(\Jwf)\leq\abs{\Awf}$.
Let $B\in\Jwf^+$.
If $B\in\Iwf$, then for $A=B$ we have $A\in\Awf$ and
$A\cap B=B\in\Jwf^+$.
Otherwise, $B\in\Iwf^*$.
Since $\Iwf{\restriction}B\nsubseteq\Jwf$, there is
$A\in(\Iwf{\restriction}B)\setminus\Jwf$ and then
$A\in\Awf$ and $A\cap B=A\in\Jwf^+$.
This proves $\zfrak(\Jwf)\leq\abs{\Pwf(X)}$.

Let $\Jwf$ be an ideal such that $\zfrak(\Jwf)$ is defined
and let a~fimily $\Awf\subseteq\Jwf^+$ witness the equality
$\zfrak(\Jwf)=\abs{\Awf}$.
Since $\Awf$ has $\Jwf$-f.u.p., there is a~maximal ideal~$\Iwf$
on~$X$ containing $\Awf\cup\Jwf$ and hence
$\Awf\subseteq\Iwf\setminus\Jwf$.
By contradiction we prove that $\Iwf\nrsub\Jwf$.
Assume that $\Iwf\rsub\Jwf$ and let $B\in\Iwf^*$ be such that
$\Iwf\cap\Pwf(B)=\Jwf\cap\Pwf(B)$.
Then $B\in\Jwf^+$ and so there is $A\in\Awf$ such that
$A\cap B\in\Jwf^+$.
On the other hand $A\cap B\in\Iwf\cap\Pwf(B)\subseteq\Jwf$.
A~contradiction.

\ref{zfrakc}
The values $\zfrak'(\Jwf)$ and $\zfrak(\Jwf)$ cannot be finite
because if $\Awf\subseteq\Jwf^+$ and
$X\setminus\bigcup\Awf\in\Jwf^+$, then $\Awf$~is not a~witness
for $\zfrak'(\Jwf)\leq\abs{\Awf}$ nor
for $\zfrak(\Jwf)\leq\abs{\Awf}$.
If $\zfrak(\Jwf)$ is defined, then also
$\zfrak'(\Jwf)$ is defined and
$\omega\leq\zfrak'(\Jwf)\leq\zfrak(\Jwf)\leq\abs{\Pwf(X)}$.
\end{proof}

Recall Rudin--Blass ordering of ideals on~$\omega$:
$\Iwf\leq_\RB\Jwf$ if there is a~finite-to-one function
$\varphi\colon\omega\to\omega$ such that
$\Iwf=\set{A\subseteq\omega}{\varphi^{-1}[A]\in\Jwf}$.

\begin{lemma}\label{RBorder}
If $\Jwf$ and $\Kwf$ are ideals on $\omega$ and $\Kwf\leq_\RB\Jwf$,
then\/ $\zfrak'(\Jwf)\leq\zfrak'(\Kwf)$.
\end{lemma}

\begin{proof}
Let $\varphi\colon\omega\to\omega$ be a~finite-to-one function and
$\Kwf=\set{A\subseteq\omega}{\varphi^{-1}[A]\in\Jwf}$.
Assume that $\Awf\subseteq\Kwf^+$ has $\Kwf$-f.u.p.\ and
$\abs{\Awf}<\zfrak'(\Jwf)$.
Let $\Awf_\varphi=\set{\varphi^{-1}[A]}{A\in\Awf}$.
Then $\Awf_\varphi\subseteq\Jwf^+$ and
$\Awf_\varphi$~has $\Jwf$-f.u.p.\ because for every
finite set $\Awf_0\subseteq\Awf$,
$\omega\setminus\bigcup\Awf_0\notin\Kwf$ and therefore
$\omega\setminus\bigcup_{A\in\Awf_0}\varphi^{-1}[A]=$
$\varphi^{-1}[\omega\setminus\bigcup\Awf_0]\notin\Jwf$.
By the cardinality assumption there is $B\in\Jwf^+$ such that
for every $A\in\Awf$,
$\varphi^{-1}[A]\cap B$ is finite.
Let
$C=\set{n\in\omega}{\varphi^{-1}[\{n\}]\cap B\ne\emptyset}$.
Then $C\in\Kwf^+$ because $\varphi^{-1}[C]\supseteq B\in\Jwf^+$.
For every $A\in\Awf$, $A\cap C$ is finite because
$\set{n\in A}{\varphi^{-1}[\{n\}]\cap B\ne\emptyset}$ is finite.
Thus, $\Awf$ cannot be a witness for $\zfrak'(\Kwf)$.
\end{proof}

Since $\fin\leq_\RB \Jwf$ for each meager ideal $\Jwf$ (see \cite{talagrand,bartjud}), we get the following consequence of \autoref{RBorder}.

\begin{corollary}
If $\Jwf$ is meager then $\zfrak'(\Jwf)\leq\pfrak$.
\qed
\end{corollary}

Note that an ideal on~$\omega$ is meager (as a~subset of $\Pwf(\omega)$) if and only if it has the Baire property.
Therefore, $\zfrak'(\Jwf)\leq\pfrak$ for every analytic ideal~$\Jwf$ on~$\omega$.
Of course, there are also non-analytic ideals satisfying this inequality.
For example, an ideal generated by an almost disjoint family is meager because it can be extended to
a~Borel ideal of the form
$\Jwf'=\set{A\subseteq\omega}
{\exists m\in\omega\allowbreak\ A\subseteq
\bigcup_{n<m}A_n}\cup\set{A\subseteq\omega}
{\forall n\in\omega\allowbreak\ |A\cap A_n|<\omega}$,
where $\set{A_n}{n\in\omega}$ is a~partition of~$\omega$ into infinite sets.

For every maximal ideal $\Iwf$ and $\Jwf=\fin$,
$\cofJ{\fin}{\Iwf}=\cof{\Iwf}\geq\ufrak$, where $\ufrak$~is the ultrafilter number.
In general, if $\Jwf\neq\fin$ and $\Iwf\supseteq\Jwf$ is maximal, we know much less about the lower bound of $\cofJ{\Jwf}{\Iwf}$.

An ideal $\Jwf$ on~$X$ is said to be \textit{nowhere maximal},
if $\Jwf{\restriction}A$ is not a~maximal ideal on~$A$ for every $A\in\Jwf^+$
(equivalently, if there is no maximal ideal $\Iwf$ on~$X$ such that
$\Iwf\rsub\Jwf$).

\begin{prop}
Let $\Jwf$ be a~nowhere maximal ideal on~$\omega$.
For every maximal ideal $\Iwf\supseteq\Jwf$,
$\cofJ{\Jwf}{\Iwf}\geq\max\{\omega_1,\zfrak(\Jwf)\}$.
\end{prop}

\begin{proof}
$\zfrak(\Jwf)$ is defined and $\zfrak(\Jwf)\geq\omega$ by \autoref{zfrak}\ref{zfrakb}--\ref{zfrakc}.
Let $\Iwf\supseteq\Jwf$ be a~maximal ideal and let
$\Awf\subseteq\Iwf$ be arbitrary set of
cardinality~$<\zfrak(\Jwf)$.
Since $\Awf$~has $\Jwf$-f.u.p., there is $B\in\Jwf^+$
such that $\forall A\in\Awf$ $A\cap B\in\Jwf$.
Since $\Jwf{\restriction}B$ is not a~maximal ideal, there
are disjoint sets $B_1,B_2\in\Jwf^+$ such that $B=B_1\cup B_2$.
By maximality of $\Iwf$, some $B_i\in\Iwf$.
Then for every $A\in\Awf$, $A\cap B_i\in\Jwf$ and hence,
$B_i\setminus A\notin\Jwf$.
This finishes the proof of $\cofJ{\Jwf}{\Iwf}\geq\zfrak(\Jwf)$.
By \autoref{maxB}\ref{maxBb}, $\cofJ{\Jwf}{\Iwf}\geq\omega_1$.
\end{proof}

An ideal $\Iwf$ is a $\mathrm{P}_\kappa(\Jwf)$-ideal iff for any $\Awf\subseteq\Iwf$ of cardinality~$<\kappa$ there is $I\in \Iwf$ such that for every $A\in\Awf$, $A\subseteq^\Jwf I$.
An ideal is a~$\mathrm{P}_\kappa$-ideal, if it is a~$\mathrm{P}_\kappa(\fin)$-ideal.

\begin{prop}\label{Pkappa}
Let $\Jwf\subseteq\Iwf$ be ideals.
\begin{enumerate}[label=\normalfont(\alph*)]
\item\label{i}
If $\Iwf$ is a~$P_\kappa(\Jwf)$-ideal, then
either $\cofJ{\Jwf}{\Iwf}=1$ or\/ $\cofJ{\Jwf}{\Iwf}\geq\kappa$.
\item\label{ii}
If $\Iwf$ is a~$\mathrm{P}_\kappa$-ideal, then for every ideal $\Jwf\subseteq\Iwf$,
either $\cofJ{\Jwf}{\Iwf}=1$ or\/ $\cofJ{\Jwf}{\Iwf}\geq\kappa$.
\end{enumerate}
\end{prop}

\begin{proof}
\ref{ii} is a~consequence of \ref{i} because
a~$\mathrm{P}_\kappa$-ideal~$\Iwf$ is a~$\mathrm{P}_\kappa(\Jwf)$-ideal for every
ideal $\Jwf\subseteq\Iwf$.
We prove \ref{i}.
Let $\Awf\subseteq\Iwf$ be arbitrary set of cardinality~$<\kappa$.
Since $\Iwf$~is a~$P_\kappa(\Jwf)$-ideal there is a~set
$A_0\in\Iwf$ such that $A\setminus A_0\in\Jwf$ for all
$A\in\Awf$.
Assume that $\cofJ{\Jwf}{\Iwf}>1$.
Then there is $B\in\Iwf$ such that $B\setminus A_0\notin\Jwf$ and
then $B\setminus A\notin\Jwf$ for all $A\in\Awf$.
This proves that $\cofJ{\Jwf}{\Iwf}\geq\kappa$.
\end{proof}

\begin{example}
By \cite{blassshelah} there is a~model
with $\dfrak=\sfrak=\cfrak=\aleph_2$ and $\bfrak=\aleph_1$
in which there is a~maximal ideal~$\Iwf$
having a~base of order type~$\omega_2$ with respect to~$\subseteq^*$.
It follows by \autoref{Pkappa} that for every ideal $\Jwf\subseteq\Iwf$, either
$\cofJ{\Jwf}{\Iwf}=1$ or $\cofJ{\Jwf}{\Iwf}=\omega_2$.
\qed
\end{example}

\section{Open problems}

Since in this article we only resolved the dichotomous division into trivial and non-trivial cases of the relative cofinality of ideals on $\unit$, there are a lot of questions regarding the studied ideals, e.g., the following.

\begin{question}
For which pairs of ideals $\Iwf$, $\Jwf$ among
$\Nwf$, $\Mwf$, $\Ewf$, $\SN$, $\nwd$ we have consistently\/
$\cofJ{\Jwf}{\Iwf}<\cof{\Iwf}$?
\end{question}

For all ideals, $\add{\Iwf}=\cof{\Iwf}$ implies
$\cofJ{\Jwf}{\Iwf}=\cof{\Iwf}$ whenever $\cofJ{\Jwf}{\Iwf}$
is infinite.
We ask the following question:

\begin{question}
For which pairs of ideals $\Iwf$, $\Jwf$ among
$\Nwf$, $\Mwf$, $\Ewf$, $\SN$, $\nwd$ we have consistently\/
$\add{\Iwf}<\cof{\Iwf}$ and $\cofJ{\Jwf}{\Iwf}=\cof{\Iwf}$?
\end{question}

The answer to the previous question when the inequality $\add{\Iwf}<\cof{\Iwf}$ is not assumed is trivial -- just consider Continuum Hypothesis or any model with $\add{\Iwf}=\cof{\Iwf}$.

\begin{question}
By \autoref{maxB}\ref{maxBb}, if $\Iwf$ is a~maximal ideal on~$\omega$ and $\Iwf\nrsub\Jwf$, then
$\cofJ{\Jwf}{\Iwf}\geq\omega_1$.
(a)
Is there a~maximal ideal~$\Iwf$ on~$\omega$ with $\cof{\Iwf}=\cfrak$
and $\cofJ{\Jwf}{\Iwf}=\omega_1$ for some ideal $\Jwf\subseteq\Iwf$?
(b)
Is there a~sequence $\seq{\Jwf_\alpha}{\alpha<\omega}$ of ideals
on~$\omega$ strictly increasing with respect to the inclusion such
that $\bigcup_{\alpha<\omega_1}\Jwf_\alpha$ is a~maximal ideal?
Note that an affirmative answer to~(a) means a~positive answer
to~(b).
\end{question}

\begin{question}
Characterize the sets of cardinals
$\Cof(\Iwf)=\set{\cofJ{\Jwf}{\Iwf}}{\Jwf\subseteq\Iwf}$
for maximal ideals~$\Iwf$ on~$\omega$ with $\cof{\Iwf}=\cfrak$.
\end{question}

{%\small
\bibliography{left}
\bibliographystyle{alpha}
}

%\Addresses

\end{document}